\documentclass[10pt,a4paper,reqno,oneside]{amsart}

\usepackage{a4wide}
\usepackage{setspace}
\usepackage[totalwidth=14cm,totalheight=22cm]{geometry}
\usepackage[T1]{fontenc}
\usepackage[dvips]{graphicx}
\usepackage{epsfig}
\usepackage{amsmath,amssymb,amscd,amsthm}
\usepackage{subfigure}
\usepackage[latin1]{inputenc}
\usepackage{latexsym}
\usepackage{graphics}
\usepackage{colortbl}
\usepackage{color}
\usepackage{ae}
\usepackage{enumitem}
\usepackage[all]{xy}
\usepackage{scalerel}
\usepackage{tikz}
\usepackage{cancel}
\usepackage{bbm,amsbsy}
\usepackage{mathrsfs}  
\usepackage{nicefrac}
\usepackage{hyperref}

    \graphicspath{{../}}
    \makeatletter
    \providecommand*{\input@path}{}
    \g@addto@macro\input@path{{../}}
    \makeatother
    \usepackage[normalem]{ulem}

\newlist{steps}{enumerate}{1}
\setlist[steps, 1]{itemsep=3pt,leftmargin=0cm,itemindent=.5cm,labelwidth=\itemindent,labelsep=0cm,align=left,label = \emph{Step \arabic*}:\,}

\makeatletter
\newcommand{\myitem}[1]{%
\item[#1]\protected@edef\@currentlabel{#1}%
}
\makeatother

\makeatletter
\newtheorem*{rep@theorem}{\rep@title}
\newcommand{\newreptheorem}[2]{%
\newenvironment{rep#1}[1]{%
 \def\rep@title{#2 \ref{##1}}%
 \begin{rep@theorem}}%
 {\end{rep@theorem}}}
\makeatother

\makeatletter
\newtheorem*{rep@cor}{\rep@title}
\newcommand{\newrepcor}[2]{%
\newenvironment{rep#1}[1]{%
 \def\rep@title{#2 \ref{##1}}%
 \begin{rep@cor}}%
 {\end{rep@cor}}}
\makeatother

\makeatletter
\newtheorem*{rep@prop}{\rep@title}
\newcommand{\newrepprop}[2]{%
\newenvironment{rep#1}[1]{%
 \def\rep@title{#2 \ref{##1}}%
 \begin{rep@prop}}%
 {\end{rep@prop}}}
\makeatother

\newtheorem{cor}{Corollary}[section]
\newtheorem{corx}{Corollary}

\newtheorem{theorem}[cor]{Theorem}
\newtheorem{thmx}[corx]{Theorem}

\newtheorem{prop}[cor]{Proposition}

\newrepcor{cor}{Corollary}
\newreptheorem{theorem}{Theorem}
\newrepprop{prop}{Proposition}

\newtheorem{lemma}[cor]{Lemma}

\newtheorem*{problem*}{Problem}

\newtheorem*{theorem*}{Theorem}
\newtheorem*{cor*}{Corollary}
\theoremstyle{definition}
\newtheorem{defi}[cor]{Definition}
\theoremstyle{remark}
\newtheorem{remark}[cor]{Remark}
\newtheorem*{remark*}{Remark}

\theoremstyle{plain}
\newcommand{\thistheoremname}{}
\newtheorem*{genericthm}{\thistheoremname}

\newcommand{\C}{{\mathbb C}}

\newcommand{\R}{{\mathbb R}}

\newcommand{\Hyp}{\mathbb{H}}

\newcommand{\tr}{\mbox{\rm tr}}

\newcommand{\D}{\mathbb{D}}

\newcommand{\I}{\mathrm{I}}
\newcommand{\II}{\mathrm{I}\hspace{-0.04cm}\mathrm{I}}
\newcommand{\III}{\mathrm{I}\hspace{-0.04cm}\mathrm{I}\hspace{-0.04cm}\mathrm{I}}

\newcounter{notes}

\DeclareRobustCommand{\SkipTocEntry}[4]{}

\begin{document}

\setcounter{secnumdepth}{3}
\setcounter{tocdepth}{2}

\title{Complete CMC hypersurfaces in Minkowski $(n+1)$-space}
\author[Francesco Bonsante]{Francesco Bonsante}
\address{Francesco Bonsante: Dipartimento di Matematica ``Felice Casorati", Universit\`{a} degli Studi di Pavia, Via Ferrata 5, 27100, Pavia, Italy.} \email{bonfra07@unipv.it} 
\author[Andrea Seppi]{Andrea Seppi}
\address{Andrea Seppi: CNRS and Universit\'e Grenoble Alpes, 100 Rue des Math\'ematiques, 38610 Gi\`eres, France.} \email{andrea.seppi@univ-grenoble-alpes.fr}
\author[Peter Smillie]{Peter Smillie}
\address{Peter Smillie: Caltech, PMA Division, 1200 E California Blvd, Pasadena, CA, 91125, USA.} \email{smillie@caltech.edu}


\thanks{The first aurthor was partially supported by Blue Sky Research project ``Analytic and geometric properties of low-dimensional manifolds" . The first two authors are members of the national research group GNSAGA}

\begin{abstract}
We prove that any regular domain in Minkowski space is uniquely foliated by spacelike constant mean curvature (CMC) hypersurfaces. This completes the classification of entire spacelike CMC hypersurfaces in Minkowski space initiated by Choi and Treibergs. As an application, we prove that any entire surface of constant Gaussian curvature in 2+1 dimensions is isometric to a straight convex domain in the hyperbolic plane.
\end{abstract}

\maketitle
%



\section*{Introduction}

The study of spacelike hypersurfaces of constant mean curvature (CMC in short) in Minkowski space $\R^{n,1}$ has been widely developed since the 1980s, see for instance \cite{Treibergs:1982aa,tkmilnor,MR680653,MR963010,choitreibergs}. 
An important motivation is that among spacelike hypersurfaces in $\R^{n,1}$, CMC hypersurfaces are precisely those for which the Gauss map, with values in the hyperbolic space $\Hyp^n$, is \emph{harmonic}. 
Employing this idea for $n=2$, many interesting results have been obtained on harmonic maps from $\C$ or $\D$ to $\Hyp^2$ (see \cite{MR963010,nishi,MR1163452,MR1216578,MR1362649,GMM}). More recently several results appeared on CMC hypersurfaces in $\R^{n,1}$ admitting a co-compact action, thus giving rise to CMC compact Cauchy hypersurfaces in certain flat Lorentzian manifolds,  in \cite{MR1957665,MR2904912}, for $n=2$ in \cite{MR1968267,MR2216148}, and for manifolds with conical singularities in \cite{Schlenker-Krasnov,Chen2019}. The generalization of this problem to general Lorentzian manifolds satisfying some additional conditions is also of importance to general relativity, for example \cite{Gerhardt:1983aa}; see \cite{Bartnik:1987aa} or Section 4.2 of \cite{Gerhardt:2006aa} for a summary.

In this paper, we focus our attention on \emph{entire} spacelike hypersurfaces in $\R^{n,1}$, that is, graphs of functions $f:\R^n\to\R$ with $|Df|<1$. Entireness is equivalent to being properly embedded (Proposition \ref{prop:prop embedded}), and thus is invariant by the action of the isometry group of $\R^{n,1}$. While the only entire hypersurfaces of \emph{vanishing} mean curvature are spacelike planes (\cite{chengyau}, also \cite{calabi} for $n \leq 4$), hypersurfaces of constant mean curvature $H\neq 0$ have a much greater flexibility, with many examples produced in \cite{Treibergs:1982aa,choitreibergs}. Still there is some rigidity: Cheng and Yau, in the same article \cite{chengyau}, show that entire CMC hypersurfaces have complete induced metric and are convex (up to applying a time-reversing isometry). 

In this paper, we first provide a complete classification of entire CMC hypersurfaces in $\R^{n,1}$ (Theorem \ref{thm: existence}). Then we derive several applications of this classification in dimension three, that is for surfaces in $\R^{2,1}$, concerning surfaces of constant \emph{Gaussian} curvature and minimal Lagrangian diffeomorphisms between simply-connected hyperbolic surfaces.

\subsection*{Classification of entire CMC hypersurfaces}
Perhaps surprisingly, although partial results were obtained in \cite{Treibergs:1982aa,choitreibergs}, to our knowledge the literature lacks a complete classification of entire CMC hypersurfaces in Minkowski space. 

The fundamental notion for our classification is the \emph{domain of dependence} $\mathcal D(\Sigma)$ of a spacelike hypersurfaces $\Sigma$. Namely, $\mathcal D(\Sigma)$ is the set of points $p\in \R^{n,1}$ such that every inextensible causal curve though $p$ meets $\Sigma$ (Definition \ref{def domain of dependence}). The domain of dependence of any entire CMC hypersurface is a \emph{regular domain} {(Proposition \ref{prop dod of CMC is regular})}, a notion introduced in \cite{Bonsante} (see also \cite{MR2110829}) meaning an open domain obtained as the intersection of at least two future half-spaces bounded by non-parallel null hyperplanes. See Section \ref{sec: preliminaries} for further definitions and explanation.
Let us now state our classification result.

\begin{thmx}\label{thm: existence}\label{thm: uniqueness}
Given any regular domain $\mathcal D$ in $\R^{n,1}$ and any $H>0$, there exists a unique entire hypersurface $\Sigma\subset\R^{n,1}$ of constant mean curvature $H$ such that the domain of dependence of $\Sigma$ is $\mathcal D$. Moreover, {as $H$ varies in $(0,+\infty)$,} the entire hypersurfaces of constant mean curvature $H$ analytically foliate $\mathcal D$.
\end{thmx}

A first simple example of a regular domain is the cone of future timelike directions from some point $p\in \R^{n,1}$, which is the intersection of all future half-spaces bounded by lightlike hyperplanes containing $p$, and is foliated by hyperboloids. A qualitatively opposite example are \emph{wedges} (Figure \ref{fig:trough}), namely those regular domains obtained as the intersection of precisely two future half-spaces neither of which is contained in the other. These are foliated by \emph{troughs}, that is, entire CMC hypersurfaces which are products of a hyperbola and a $(n-1)$-dimensional spacelike affine subspace.

\begin{figure}[htb]
\centering
\includegraphics[height=4.5cm]{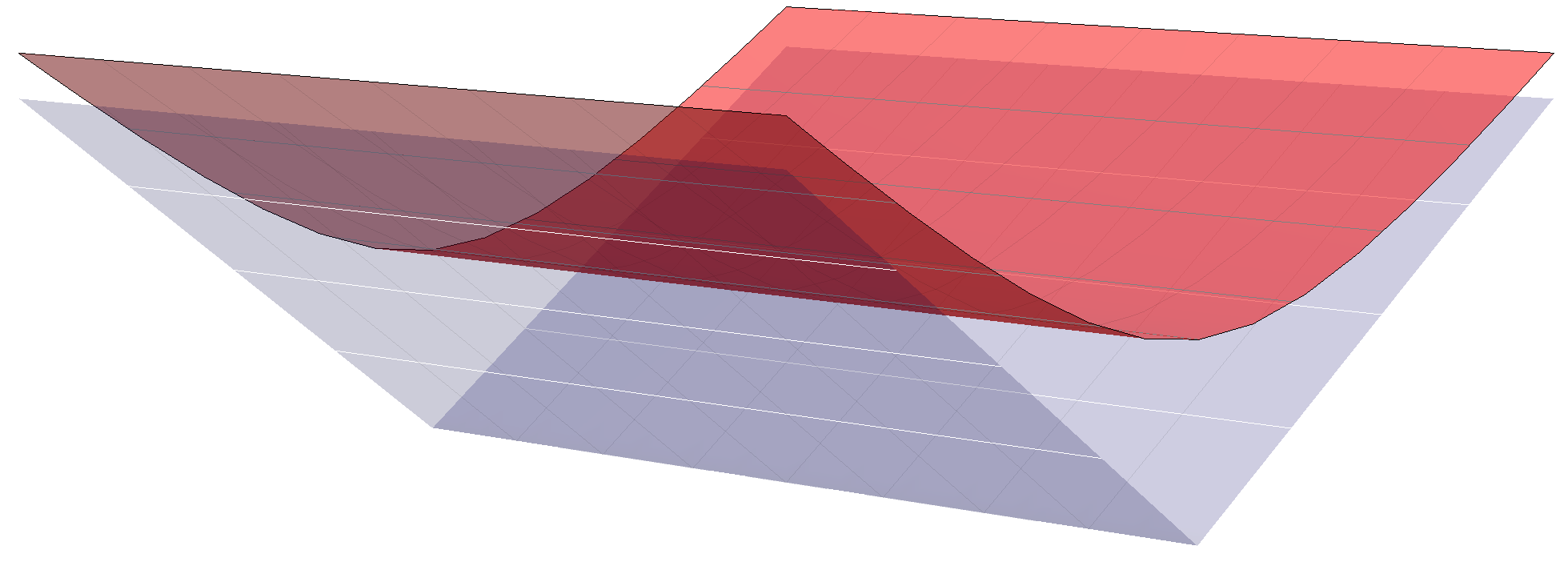}
\caption{The two dimensional trough $T$, whose domain of dependence is a wedge. \label{fig:trough}}
\end{figure}

There is a 1-to-1 correspondence between regular domains in $\R^{n,1}$ and lower semicontinuous functions $\varphi:\mathbb S^{n-1}\to\R\cup\{+\infty\}$ (Proposition \ref{prop: support function}). The correspondence associates to the function $\varphi$ the regular domain which is obtained as the intersection of the half-spaces $\{(\mathbf{x},x_{n+1})\in\R^{n,1}\,:\,x_{n+1}>\mathbf{x}\cdot \mathbf{y}-\varphi(\mathbf{y})\}$, as $\mathbf{y}$ varies in $\mathbb S^{n-1}$. For instance, the hyperboloid centered at the origin corresponds to $\varphi\equiv 0$, whereas wedges correspond to functions $\varphi$ which are finite on exactly two points. 
From this perspective, our classification of entire CMC hypersurfaces can be interpreted as follows.

\begin{thmx} There is a bijective correspondence between the set of entire CMC hypersurfaces in $\R^{n,1}$ and the set of lower semicontinuous functions on $\mathbb S^{n-1}$ finite on at least two points.
\end{thmx}

In \cite{bon_smillie_seppi} we refer to the lower semicontinous function $\varphi$ as the \emph{null support function} of the CMC hypersurface. There are at least two other notions of asymptotics of an entire surface in $\R^{n,1}$ in the literature: cuts at future null infinity as in \cite{Andersson:1999aa,Stumbles:1981aa}, and blowdown data $(L,f_0)$ as in \cite[Theorem 6.2]{choitreibergs}. In Minkowski space, all three of these notions are equivalent (Propositions \ref{prop: null cut} and \ref{prop: choi treibergs function}).

The result of \cite{choitreibergs} is an important predecessor to our theorem. To translate their result into the language of null support functions, say that a function valued in $\R \cup \{+\infty\}$ is nearly continuous if the set on which it is finite is closed and it continuous when restricted to that set. Then Choi and Treibergs prove that if $\varphi$ is a lower semicontinuous function on $\mathbb S^{n-1}$ which is nearly continuous and finite on at least two points, then there exists an entire CMC hypersurface with null support function $\varphi$. Compared to \cite{choitreibergs}, our contribution is to extend the existence theorem to all lower semicontinuous functions finite on at least two points (Section \ref{sec: existence}) and crucially to prove uniqueness (Section \ref{sec: uniqueness}).

Let us now briefly discuss the ingredients involved, starting with the proof of uniqueness.

\subsubsection*{Uniqueness} 
The proof of the uniqueness statement of Theorem \ref{thm: existence} consists in an application of the Omori-Yau maximum principle. In fact, in Theorem \ref{thm: comparison} we prove a comparison principle: if $\Sigma_1$ and $\Sigma_2$ are two entire CMC hypersurfaces with constant mean curvature $H_1\geq H_2$ and $\Sigma_2\subset \mathcal D(\Sigma_1)$, then  $\Sigma_2$ cannot meet the past of $\Sigma_1$. The uniqueness statement then follows immediately, for if $\Sigma_1$ and $\Sigma_2$ have the same constant mean curvature and the same domain of dependence, then they necessarily coincide.

To prove such a comparison result, we consider the Lorentzian distance (Definition \ref{def: Lorentzian distance}) to $\Sigma_1$ as a function $u$ on $\Sigma_2$. Where $u$ is positive, we derive the estimate (Lemma \ref{lem: subsolution}) \[\Delta u \geq \frac{nH_1H_2u}{1-H_1 u} - \frac{n H_2 |\nabla u|}{1 - H_1 u} - \frac{|\nabla u|^2}{u}~.\]
This shows immediately that $u$ cannot attain a positive maximum on $\Sigma_2$. To prove the stronger result that $u$ can never be positive at all on $\Sigma_2$, we apply the Omori-Yau argument. Namely, we observe that $u$ is bounded from above and that $\Sigma_2$ has an a priori lower bound for its Ricci curvature. This together with the key result of Cheng and Yau that $\Sigma_2$ is complete allows us to construct a supersolution of the same equation in terms of the intrinsic distance on $\Sigma_2$ which touches $u$ from above at a point. Since $u$ is a subsolution, this contradicts the maximum principle, and the contradiction shows that $u$ {cannot be positive anywhere}.



\subsubsection*{A general comparison principle}  We also include in Section \ref{sec: uniqueness} a generalization of the comparison result beyond what we need for the proof of Theorem \ref{thm: existence}. Namely, we relax the assumption of constant mean curvature to merely bounded mean curvature: if $\Sigma_1$ and $\Sigma_2$ are entire spacelike hypersurfaces with the mean curvature of $\Sigma_1$ bounded below by some positive constant $H$ and the mean curvature of $\Sigma_2$ bounded above by $H$, and if furthermore $\Sigma_2 \subset \mathcal{D}(\Sigma_1)$, then $\Sigma_2$ lies weakly in the future of $\Sigma_1$. The essence of the proof is simply to show that the entire hypersurface of constant mean curvature $H$ in $\mathcal{D}(\Sigma_1)$ lies between them. The proof of this general comparison principle thus relies on the existence part.

\subsubsection*{Existence} The ingredients for the proof of the existence of an entire CMC hypersurface in any domain of dependence are mostly contained in the articles \cite{Treibergs:1982aa,choitreibergs}. In fact, if we fix a constant $H>0$, writing an entire hypersurface as the graph of some function $u:\R^n\to\R$, the CMC condition translates to a certain quasi-linear PDE on $u$. The fundamental proposition, stated in \cite[Proposition 6.1]{choitreibergs}, asserts that if one has two functions $v,w\in C^{0,1}(\R^n)$ which are respectively a weak sub- and super-solution  for such a quasi-linear equation with $v\leq w$, then there exists a solution $u$ which is sandwiched between $v$ and $w$.
Although stated in \cite[Proposition 6.1]{choitreibergs}, the cited references \cite{Treibergs:1982aa} and \cite{MR680653} for this proposition do not provide the statement exactly in this form. For this reason, we decided to include {in Section \ref{ss:ctargument} a roadmap to the proof} for convenience of the reader.
 
Applying the above proposition, we use level sets of cosmological time for any regular domain $\mathcal{D}$ as upper and lower barriers to prove the existence part of Theorem \ref{thm: existence}. The cosmological time, $T$, is the function on a regular domain measuring the Lorentzian distance to its boundary, and its relevant properties were described in \cite{Bonsante} (see Proposition \ref{prop:T level sets}). This idea was used in the cocompact case in \cite{MR2216148}, in which the author use the hypersurface $T^{-1}(1/H)$ as a supersolution and $T^{-1}(1/nH)$ as a subsolution. To save some effort, we use $T^{-1}(0)$, the boundary of the domain of dependence, as a subsolution in our proof of existence. This is sufficient to produce an entire CMC hypersurface $\Sigma$ with $\mathcal{D}(\Sigma) = \mathcal{D}$. 


\subsubsection*{Foliation}
It remains to discuss the proof of the fact that the hypersurfaces $\Sigma_H$ having constant mean curvature $H$ and domain of dependence $\mathcal D$ foliate $\mathcal D$ itself. By the comparison theorem which we used to prove uniqueness (Theorem \ref{thm: comparison}), we obtain that the $\Sigma_H$ are pairwise disjoint, and moreover $\Sigma_{H_1}$ is in the past of $\Sigma_{H_2}$ if $H_1>H_2$. It thus remains to show that every point $p\in\mathcal D$ belongs to some $\Sigma_H$, which can be done by rather standard techniques as in \cite{MR2904912,Bonsante:2015vi,MR3789829,nieseppi}. 
In fact, given any point $p\in\mathcal D$, by techniques similar to those we used for the existence part, one shows that the two hypersurfaces defined as the supremum (resp. infimum) of all CMC hypersurfaces in $\mathcal D$ whose future (resp. past) contain $p$ are CMC hypersurfaces with the same constant mean curvature, hence by uniqueness they necessarily coincide and contain $p$ itself. The existence of \emph{some} such CMC hypersurfaces having $p$ in their future/past follows in one case from a simple upper bound on the cosmological time, and in the other from an application of the comparison principle (Theorem \ref{thm: comparison}) using troughs as barriers.

To prove that the foliation is analytic, we apply the analytic inverse function theorem in Banach spaces. To set this up, we fix a leaf $\Sigma$, and consider normal graphs of functions $u$ over $\Sigma$. The mean curvature of the graph of $u$ defines a differential operator on $\Sigma$, which we show by the inverse function theorem is locally invertible near $u=0$ as a map between global H\"older spaces. Consequently, for values of $H'$ near the mean curvature $H$ of $\Sigma$, there is a unique bounded function $u_{H'}$ on $\Sigma$ whose normal graph has mean curvature $H'$. Of course, we already knew this much from the existence of the foliation and the observation that two CMC surfaces share a domain of dependence if and only if they are a bounded distance apart. But since the mean curvature is an analytic differential operator, the analytic inverse function theorem implies that this family of H\"older functions is analytic in the parameter $H'$. Then it follows from classical results on analytic functions that in fact $u_{H'}(x)$ is jointly analytic in $H'$ and $x$, and therefore gives an analytic foliation chart.

%

\subsection*{Applications to hyperbolic surfaces in $\R^{2,1}$}
The final section of this paper focuses on $n=2$, and provides a number of applications of Theorem \ref{thm: existence} 
to surfaces of constant \emph{Gaussian} curvature, in other words surfaces such that the determinant of the shape operator is constant. Taking the constant to be one, by Gauss' equation these surfaces are \emph{hyperbolic}, meaning that the first fundamental form is a hyperbolic metric. The relationship lies in the classical observation that if $\Sigma$ has constant intrinsic curvature -1, then the surface which lies at Lorentzian distance one from $\Sigma$ to the convex side has constant mean curvature $H = 1/2$.

Just as CMC hypersurfaces are characterized as those with harmonic Gauss map, among immersed spacelike surfaces $\Sigma$ in $\R^{2,1}$ \emph{hyperbolic} surfaces are exactly those whose Gauss map $G$ is a minimal Lagrangian local diffeomorphism; that is, the graph of $G$ is a minimal Lagrangian surface in $\Sigma\times\Hyp^2$. If moreover $\Sigma$ is embedded, then $G$ is a diffeomorphism onto its image.

A classification of \emph{entire} surfaces of constant Gaussian curvature has been completed in \cite{bon_smillie_seppi}, after several partial results had been obtained in \cite{Li,schoenetal,barbotzeghib,Bonsante:2015vi}. In short, in \cite{bon_smillie_seppi} we proved that every regular domain which is the intersection of at least three pairwise non-parallel future half-spaces contains a unique entire surface $\Sigma$ of constant Gaussian curvature $K$, for any $K>0$. However, it has been observed (for instance in \cite{hanonomizu}) that an entire surface of constant Gaussian curvature $K>0$ is not necessarily complete, thus highlighting a substantial difference with respect to mean curvature. In other words the first fundamental form of $\Sigma$, being hyperbolic, is \emph{locally} isometric to $\Hyp^2$, but in general not \emph{globally} isometric.

There are thus several questions which arise naturally. For instance:
\begin{enumerate}
\myitem{$i)$}\label{question: complete} When is an entire hyperbolic surface $\Sigma$ in $\R^{2,1}$ complete, in terms of the domain of dependence of $\Sigma$?
\myitem{$ii)$}\label{question: intrinsic} When $\Sigma$ is not complete, to which hyperbolic surface is it intrinsically isometric?
\myitem{$iii)$} \label{question: intrinsic2} 
Conversely, which hyperbolic surfaces can be isometrically embedded in $\R^{2,1}$ with image an entire surface?
\end{enumerate}

Question \ref{question: complete} appears to be the most difficult, and is left for future investigation. In this paper we answer questions \ref{question: intrinsic} and \ref{question: intrinsic2}.

\subsubsection*{Entireness and minimal Lagrangian graphs}
 Let us first introduce a definition. Let $(S,h)$ and $(S',h')$ be simply connected hyperbolic surfaces. We say that
  a smooth map $f:(S,h)\to (S',h')$ is \emph{realizable} in $\R^{2,1}$ if there exists an isometric immersion $\sigma:(S,h)\to\R^{2,1}$ and a local isometry $d:(S',h')\to\Hyp^2$ such that $d\circ f= G_\sigma$ where $G_\sigma:S\to\Hyp^2$ is the Gauss map of $\sigma$. If moreover the immersion is proper, which is equivalent to its image being entire (Proposition \ref{prop:prop embedded}), we say that $f$ is \emph{properly realizable}.

It is known that realizability of $f$ is equivalent to being a minimal Lagrangian local diffeomorphism. The following theorem gives a characterization of properly realizable minimal Lagrangian maps, in terms of their graphs in the Riemannian product of $(S,h)$ and $(S',h')$.

\begin{thmx}\label{cor graph metric}
Let $f:(S,h)\to (S',h')$ be a diffeomorphism between simply connected hyperbolic surface. Then $f$ is properly realizable in $\R^{2,1}$ if and only if the graph of $f$ is a complete minimal Lagrangian surface in $(S\times S', h\oplus h')$. In this case, both $(S,h)$ and $(S',h')$ are isometric to straight convex domains in $\Hyp^2$.
 \end{thmx}

The second part of the statement answers question \ref{question: intrinsic}. A \emph{straight convex domain} in $\Hyp^2$ is the interior of the convex hull of a subset of $\partial_\infty\Hyp^2$ consisting of at least 3 points. See also Corollary \ref{cor first fundamental form} below.

Observe that from the definition, it is easy to see that the inverse of a minimal Lagrangian diffeomorphism is again minimal Lagrangian. The following is then a straightforward corollary of Theorem \ref{cor graph metric}:

\begin{corx}\label{thm:inverse min lag}
Let $f:(S,h)\to (S',h')$ be a minimal Lagrangian diffeomorphism between simply connected hyperbolic surface. Then $f$ is properly realizable in $\R^{2,1}$ if and only if $f^{-1}$ is properly realizable in $\R^{2,1}$.
\end{corx}

\subsubsection*{Outline of the proof} Let us spend a few words here to outline the proof of Theorem \ref{cor graph metric}.
The basic observation (Proposition \ref{prop normal flow}) is that for any entire hyperbolic surface $\Sigma$ in $\R^{2,1}$, the surface $\Sigma_+$ at Lorentzian distance one with constant mean curvature $H= 1/2$ is still entire, and the two have the same domain of dependence. A consequence of the uniqueness of Theorem \ref{thm: uniqueness}, together with the main result of \cite{bon_smillie_seppi}, is that the converse is almost always true (Corollary \ref{cor past flow uniqueness}): if $\Sigma_+$ is any entire CMC-\nicefrac{1}{2} surface except for the trough, then the surface $\Sigma$ at Lorentzian distance one to the past is still entire (with the same domain of dependence).
To prove the first part of Theorem \ref{cor graph metric}, it then suffices to observe that the first fundamental form of $\Sigma_+$ is bi-Lipschitz equivalent to the induced metric on the graph of the minimal Lagrangian map in the Riemannian product, and by the Cheng and Yau completeness theorem, entireness of the equidistant CMC-\nicefrac{1}{2} surface $\Sigma_+$ is equivalent to completeness of its first fundamental form.


The second part of Theorem \ref{cor graph metric} is proved by applying \cite[Theorem E]{bon_smillie_seppi}, which states that the image of the Gauss map of any entire hyperbolic surface is a straight convex domain. Alternatively one can apply a similar statement for entire CMC hypersurfaces given in \cite[Theorem 4.8]{choitreibergs}). The symmetry provided by Corollary \ref{thm:inverse min lag} then implies that $(S,h)$ is isometric to a straight convex domain as well.

\subsubsection*{Characterizing the intrinsic metrics}
Let us now conclude by answering question \ref{question: intrinsic2}. 

\begin{corx} \label{cor first fundamental form}
A hyperbolic surface can be embedded isometrically and properly in $\R^{2,1}$ if and only if it is isometric to a straight convex domain.
\end{corx}

Being a necessary condition follows from Theorem \ref{cor graph metric}. To show that the condition is also sufficient, \cite[Theorem A]{bon_smillie_seppi} implies that one can find an entire hyperbolic surface whose Gauss map has image any straight convex domain. Applying again  Corollary \ref{thm:inverse min lag}
 gives the desired statement.

As a final comment, the hypothesis of entireness is clearly essential in Corollary \ref{cor first fundamental form}, as any domain in $\Hyp^2$ can be realized without the entireness assumption. But we remark here that the situation is  even subtler, since also hyperbolic surfaces which are not isometric to a subset of $\Hyp^2$ can be embedded as non-entire surfaces.
In fact, in \cite[Appendix A]{Bonsante:2015vi}, an example of non-entire surface in $\R^{2,1}$ intrinsically isometric to the universal cover of the complement of a point in $\Hyp^2$ is constructed.


 \subsection*{Organization of the paper}
 
 In Section \ref{sec: preliminaries} we introduce the necessary background, and in addition we show the equivalence of several notions of asymptotics. In Section \ref{sec: uniqueness} we prove the uniqueness part of Theorem \ref{thm: uniqueness}, while Section \ref{sec: existence} shows the existence part and Section \ref{sec: foliation} shows the foliation result. Finally, Section \ref{sec: application} gives applications in dimension $2+1$. 
 
\subsection*{Acknowledgements}

The third author would like to thank Jonathan Zhu and Lu Wang for helpful conversations.

\section{Preliminaries} \label{sec: preliminaries}

\subsection{Causality and Entire hypersurfaces}

Minkowski $(n+1)$-space is the Lorentzian manifold $\R^{n,1} = (\R^n, dx_1^2 + \cdots + dx_n^2 - dx_{n+1}^2)$. We say that a vector is \emph{spacelike} if its square norm is positive, \emph{timelike} if its square norm is negative, and \emph{null} if its square norm is zero. A subspace of $\R^{n,1}$ is \emph{spacelike}, \emph{timelike}, or \emph{null} if the restriction of the inner product to it is Euclidean, Lorentzian, or degenerate respectively. We say a timelike or null vector is future oriented if its last coordinate is positive, and past oriented if it is negative. If $p \in \R^{n,1}$ we define the \emph{future}, $I^+(p)$, to be the set of points $p + v$ for $v$ a future oriented timelike vector, and similarly for the past, $I^-(p)$. If $X$ is a set in $\R^{n,1}$, define $I^\pm(X) = \cup_{p \in X} I^\pm(p)$. We also define the \emph{causal future} $J^+(p)$ and $J^+(X)$ the same way, except that we allow the vector $v$ to be timelike or null. Since the zero vector is null, $X \subset J^+(X)$.

A $C^0$ curve in $\R^{n,1}$ is \emph{causal} if each pair of points on it is timelike- or null-separated. A set is \emph{achronal} if each pair of points on it is spacelike- or null-separated. An achronal hypersurface will mean a connected $C^0$ hypersurface which is achronal. Note that a causal curve is locally the graph of a 1-Lipschitz function $\R \to \R^n$, where we decompose $\R^{n+1} = \R^n \oplus \R$, and an achronal surface is locally the graph of a 1-Lipschitz function $\R^n \to \R$. A causal curve is \emph{inextendable} if it is globally the graph of a 1-Lipschitz function, and an achronal surface is \emph{entire} if it is globally the graph of a 1-Lipschitz function. By a \emph{spacelike} hypersurface, we will mean a smooth hypersurface whose tangent space at each point is spacelike, so that it inherits a Riemannian metric. It is easy to verify that an entire spacelike hypersurface is achronal.

The following proposition implies that for a spacelike hypersurface, entire, properly embedded, and properly immersed are all equivalent.

\begin{prop}[{\cite[Proposition 1.10]{bon_smillie_seppi}}] \label{prop:prop embedded}
If a spacelike hypersurface is $\R^{n,1}$ is properly immersed, then it is entire.
\end{prop}

Another condition that implies entireness is completeness of the induced metric:

\begin{prop}[{\cite[Lemma 3.1]{Bonsante}}] \label{lemma complete implies entire} 
Let $\sigma:M^n\to \R^{n,1}$ be a spacelike immersion such that the first fundamental form is a complete Riemannian metric. Then $\sigma$ is an embedding and its image is an entire hypersurface.
\end{prop}

As mentioned in the introduction, the converse of this second proposition is not true without some curvature assumptions; it is easy to construct examples of entire spacelike hypersurfaces such that the induced metric is not complete.


\subsection{Domains of dependence and regular domains}

Our tool for understanding the asymptotics of entire spacelike hypersurfaces will be their domain of dependence. This gives a fairly coarse notion of asymptotics, but it turns out to be exactly what we need for the classification of entire CMC hypersurfaces. References for Propositions 1.5 and 1.6 can be found, with some adaptation, in Section 6.5 of \cite{Hawking:1973aa} or presented in a slightly different order in \cite{bon_smillie_seppi}.

\begin{defi} \label{def domain of dependence}
For an achronal set $\Sigma$ in $\R^{n,1}$, its \emph{domain of dependence} $\mathcal{D}(\Sigma)$ 
is the set of points $p \in\R^{n,1}$ such that every inextendable causal curve through $p$ meets $\Sigma$.
\end{defi}

\begin{defi}
An achronal set $\mathcal{H}$ is a \emph{past horizon} if for any point $p \in \mathcal{H}$, there is a future directed (hence nonzero) null vector $v$ such that $p + v$ is still in $\mathcal{H}$. 
\end{defi}

We note that the empty set is a past horizon. We furthermore define \emph{the} past horizon $\mathcal{H}^-(\Sigma)$ of an achronal set $\Sigma$ to be the part of the boundary of the domain of dependence of $\Sigma$ which lies in the past of $\Sigma$. The compatibility of this terminology is guaranteed by the following proposition.

\begin{prop} The past horizon $\mathcal{H}^-(\Sigma)$ of an achronal set $\Sigma$ is a past horizon. Moreover, every past horizon is the past horizon of itself.
\end{prop}

We define the future horizon $\mathcal{H}^+(\Sigma)$ analogously, but we will focus on past horizons in this paper. If $\Sigma$ is an entire achronal hypersurface, then its past horizon is either empty or is itself entire. We state some elementary properties of entire past horizons:

\begin{prop} \label{prop: entire past horizons}
Let $\mathcal{H}$ be an entire past horizon. Then
\begin{itemize}
\item $\mathcal{H}$ is convex.
\item If $p \in \mathcal{H}$ and $v$ is a future null vector such that $p + v \in \mathcal{H}$, then the entire geodesic ray $\{p + t v \,|\, t \in [0, \infty)\}$ is contained in $\mathcal{H}$.
\item $\mathcal{H}$ is the envelope of its null support planes.
\end{itemize}
\end{prop}

Since $\mathcal{H}$ is a convex graph, it is determined by its locus of support planes. If $\mathcal{H}$ is the graph of $f$, the locus of support planes is described by the Legendre transform $f^*: \R^n \to \R \cup \{+\infty\}$, defined by $f^*(\mathbf y) = \sup_{\mathbf y \in \R^n} \mathbf x \cdot \mathbf y - f(\mathbf y)$. In general, the Legendre transform $f^*$ is a lower semicontinuous function which may take the value $+ \infty$. The third point of the proposition says that $\mathcal{H}$ is determined by the restriction of $f^*$ to the unit sphere, which corresponds to the null support planes. We summarize this as:

\begin{prop} \label{prop: support function}
Past horizons are in bijection with lower semicontinous functions on the sphere taking values in $\R \cup \{ + \infty \}$.
\end{prop}

This lower semi-continuous function is called the \emph{null support function} of the past horizon. In fact, one may speak of the null support function of any entire achronal set, meaning simply the null support function of its past horizon. 

If $\Sigma$ is a spacelike hypersurface, then the following proposition states that its domain of dependence is really a domain (i.e. it is open).

\begin{prop}{\cite[Lemma 1.14]{bon_smillie_seppi}} \label{prop: D is open} If $\Sigma$ is a spacelike hypersurface, then 
\begin{itemize}
\item For any $p \in \mathcal{D}(\Sigma)$, there is a compact subset $K \subset \Sigma$ such that $p \in \mathcal{D}(K)$.
\item $\mathcal{D}(\Sigma)$ is open.
\end{itemize}
\end{prop}

If $\Sigma$ is an entire spacelike hypersurface, then $\mathcal{H}^{\pm}$ are entire or empty, and it follows from the proposition that $\mathcal{D}(\Sigma)$ is the open region between them. A case of particular interest is when $\Sigma$ is an entire convex spacelike surface. For us, a convex hypersurface will always mean one that is the graph of a convex function (in particular, the past connected component of the two sheeted hyperboloid is not called convex). For a convex hypersurface $\Sigma$, it is not hard to see that its future horizon is empty, so that $\mathcal{D}(\Sigma) = I^+(\mathcal{H}^-(\Sigma))$ if $\mathcal{H}^-(\Sigma)$ is nonempty, and $\mathcal{D}(\Sigma) = \R^{n,1}$ otherwise.

Having broken the time reversal symmetry, we make the following definition.
\begin{defi} \label{def: regular domain}
A \emph{regular domain} is an open domain which is the future of an entire past horizon with at least one spacelike support plane.
\end{defi}
Equivalently, it must have at least two non-parallel null support planes. Under the correspondence between entire past horizons and lower semicontinous functions on the sphere, this just excludes the function that is identically equal to $+ \infty$ and functions which are finite at a single point. 

A regular domain has an important canonical function on it called cosmological time, which we now discuss.

\begin{defi} \label{def: Lorentzian distance} For $q$ in the causal future of $p$ (written $q \in J^+(p)$), define the Lorentzian distance $d(p,q) = \sqrt{-\langle q-p, q-p\rangle}$. More generally, if $\Sigma_1$ and $\Sigma_2$ are two achronal sets such that there exists at least one future directed causal geodesic from a point in $\Sigma_1$ to a point in $\Sigma_2$, define the Lorentzian distance, which may be infinite, by
\[d(\Sigma_1, \Sigma_2) = \sup_{\substack{p \in \Sigma_1 \\ q \in J^+(p) \cap \Sigma_2}} d(p,q)~.\]
\end{defi}

If $\mathcal{D}$ is a regular domain with past boundary $\mathcal{H}$, the \emph{cosmological time} is a function $T$ defined on $\mathcal{D}$ by $T(p) = d(\mathcal{H}, p)$. The assumption that $\mathcal{H}$ has at least one spacelike support plane guarantees that $T$ is finite. More generally, we have

\begin{prop} [{\cite[Proposition 4.3 and Corollary 4.4]{Bonsante}}] \label{prop:T level sets}
Let $\Sigma$ be a convex entire achronal hypersurface with at least one spacelike support plane.
\begin{itemize}
\item The function $T(p) = d(\Sigma, p)$ is a $C^1$ function on $I^+(\Sigma)$.
\item The function $T$ tends to zero as $p$ approaches $\Sigma$. It is convex and unbounded along any timelike geodesic.
\item The level sets $T^{-1}(r)$ for $r > 0$ are convex entire spacelike hypersurfaces, each of which has the same domain of dependence as $\Sigma$.
\end{itemize}
\end{prop}

This proposition applies in particular to the case where $\Sigma = \mathcal{H}$ is a past horizon, $I^+(\mathcal{H})$ is a regular domain, and $T$ is the cosmological time.

\subsection{CMC hypersurfaces}

Any spacelike hypersurface $\Sigma$ has a future unit normal vector field which we will call $\nu$. Parallel transporting the vector field $\nu$ to the origin gives the \emph{Gauss map} $G: \Sigma \to \Hyp^n$, where $\Hyp^n$ is identified with the component of the hyperboloid of future unit timelike vectors.

The shape operator of $\Sigma$ is denoted $B = d\nu$, viewed as an endomorphism of the tangent bundle, and we define the mean curvature with the convention
\[
H = \frac{1}{n} \mathrm{tr}B~.
\]
We will use the notation $\I, \II, \III$ for the first, second, and third fundamental forms: $I$ is the induced metric, $\II(\cdot, \cdot) = \I(B\cdot, \cdot)$, and $\III(\cdot, \cdot) = \I(B \cdot, B \cdot)$.

The following classical theorem holds in $\R^{n,1}$ just as in Euclidean space:
\begin{theorem}[see {\cite[Theorem 1.2]{choitreibergs}}] \label{lemma harmonic}
Let $\Sigma$ be a spacelike hypersurface in $\R^{n,1}$ and let $\I$ be its first fundamental form. Then the Gauss map $G:(\Sigma,\I)\to \Hyp^n$ is harmonic if and only if $\Sigma$ has constant mean curvature.
\end{theorem}

The foundational result about spacelike hypersurfaces with constant mean curvature in Minkowski space is the following:

\begin{theorem}[\cite{chengyau}] \label{thm: Cheng-Yau}
If $\Sigma \subset \R^{n,1}$ is an entire spacelike hypersurface with constant mean curvature then $\Sigma$ is intrinsically complete with non-positive Ricci curvature.
\end{theorem}

Two comments about this theorem are in order. First, on the question of completeness, the result of Cheng and Yau is somewhat stronger: if instead of constant mean curvature we assume only a $C^1$ bound on the mean curvature function, then $\Sigma$ is still complete. Second, non-positive Ricci curvature is equivalent to convexity, as we now explain.

The Gauss equation for a spacelike hypersurface in $\R^{n,1}$ with second fundamental form $\II$ reads
\[
\mathrm{R}_{abcd} = -(\II_{ac}\II_{bd} - \II_{ad}\II_{bc})~,
\]
and tracing once, with $H = \mathrm{tr}(\II)/n$, gives either of the equivalent equations
\begin{align*}
\mathrm{Ric}_{ab}&= -(nH\II_{ab} - \III_{ab}) \\
\mathrm{Ric}_a^b &= -(nHB_a^b - B^b_cB^c_a)~.
\end{align*}
This shows that the second fundamental form and the Ricci tensor are simultaneously diagonalizable. Moreover, if $\lambda$ is an eigenvalue of $B$, the corresponding eigenvalue $\mu$ of $\mathrm{Ric}$ is given by
\[
\mu = \lambda^2 - n H \lambda~.
\]
We see that the Ricci curvature is nonpositive if and only if every eigenvalue of $B$ is between 0 and $nH$. Since the sum of the eigenvalues is $nH$, this is the same as saying that every eigenvalue is at least 0. Hence, $\Sigma$ is convex, up to time reversal. Furthermore, going back to the untraced Gauss equation, we see that nonpositive Ricci curvature implies nonpositive sectional curvature. 
We also observe that with or without the nonpositivity hypothesis, the smallest $\mu$ can be is $-n^2H^2/4$. We record these facts for later application.

\begin{prop} \label{prop: Ricci lower bound}
If $\Sigma$ is a spacelike hypersurface with mean curvature $H$ at a point $p$, then its Ricci curvature at $p$ is bounded below by $-n^2H^2/4$ times the metric.
\end{prop}

\begin{prop} \label{prop: nonpos sec curv}
If $\Sigma$ is a spacelike hypersurface with non-positive Ricci curvature, then it has non-positive sectional curvature.
\end{prop}

We will also need the following splitting theorem.

\begin{theorem} [{\cite[Theorem 3.1]{MR1216578}}] \label{thm:splitting theorem}
Suppose that $\Sigma$ is an entire hypersurface in $\R^{n,1}$ with constant mean curvature $H$, and second fundamental form $\II$. If there is a point $p \in \Sigma$ and a tangent vector $v \in T_p\Sigma$ such that $\II(v,v) = 0$, then $\Sigma$ splits extrinsically as the product of a line and an $n-1$ dimensional submanifold $\Sigma'$. In other words, there is a hypersurface $\Sigma' \subset v^\perp \cong \R^{n-1,1}$ of constant mean curvature $\frac{nH}{n-1}$ such that $\Sigma = \Sigma' \times \R v$.
\end{theorem}

As the only entire CMC hypersurface in $\R^{1,1}$ is the hyperbola, a consequence of this theorem is that every entire CMC surface in $\R^{2,1}$ which is not a trough has positive definite second fundamental form everywhere.

Finally, we state here for reference a special case of Corollary \ref{cor: dom of dep regular domain}, which we will prove later.
\begin{prop} \label{prop dod of CMC is regular}
If $\Sigma$ is entire with constant mean curvature $H$, positive with respect to its future unit normal, then $\mathcal{D}(\Sigma)$ is a regular domain.
\end{prop} 

\subsection{Asymptotics}

We end the preliminary section by comparing the null support function with the other two notions of asymptotics of an entire spacelike hypersurface that appear in the literature. We start by introducing the data $(L, f_0)$ used in \cite{choitreibergs}. Here $L$ is a closed subset of $\mathbb S^{n-1}$ and $f_0$ is a function on $L$. Given a entire spacelike hypersurface expressed as the graph of a function $u: \R^n \to \R$, define
\begin{align*}
L = \{\theta \in &\mathbb S^{n -1} | \lim_{r \to +\infty} \tfrac{u(r\theta)}{r} = 1 \} \qquad \textrm{and}\\
f_0(\theta) = &\lim_{r \to +\infty} r- u(r\theta)  \qquad\qquad \textrm{for $\theta \in L$.}
\end{align*}
We remark that $L$ is closed: indeed, if we define $V(\theta) = \lim_{r \to +\infty} \frac{u(r\theta)}{r}$, then $V$ is the limit of 1-Lipschitz functions, so it is continuous, and $V^{-1}(1)$ is closed. Also, $f_0$ may in general take the value $+\infty$. 

We now show that the data $(L, f_0)$ determines the null support function $\varphi$, and so long as the mean curvature is bounded below, $\varphi$ determines $(L, f_0)$. {Recall that the null support function is defined on $\mathbb S^{n-1}$ as
$\varphi(\theta)=\sup_{\mathbf y \in \R^n} \langle\theta , \mathbf y\rangle - u(\mathbf y)$.}

\begin{prop} \label{prop: choi treibergs function} Let $u$ be a function on $\R^n$ whose graph is entire and spacelike, let $f_0$ and $L$ be defined as above, and let $\varphi$ be the null support function of $u$. 
Then 
\[
\varphi(\theta) = \begin{cases}
	f_0(\theta)& \text{for } \theta \in L\\
	+ \infty& \text{for } \theta \notin L
\end{cases}
\]
Moreover, if the graph of $u$ has mean curvature bounded below by a positive constant $H$, then $L$ is the closure of the set $\{\theta \,|\, \varphi(\theta) < + \infty\}$.
\end{prop}
\begin{proof}(See also Section 2.3 of \cite{Bonsante:2015vi}) 
First note that since $V(\mathbf x) := \lim_{r\to +\infty} \frac{u(r \mathbf x)}{r}$ is 1-Lipschitz with $V(0) = 0$, its value is at most 1 at all $\theta \in \mathbb S^{n-1}$, so we have $V(\theta) < 1$ for $\theta \notin L$. It is harmless to extend the definition of $f_0$ to all $\theta$, in which case by the previous sentence we see that $f_0(\theta) = + \infty$ for $\theta \notin L$. We now show that $f_0 = \varphi$.

Since $u$ is strictly 1-Lipschitz, the function $r - u(r\theta)$ is an increasing function of $r$, so we can replace the ``limit'' in the definition of $f_0$ with a supremum over $r$. Since the function $r - u(r\theta)$ is the restriction of $\langle \theta, \mathbf x \rangle - u(\mathbf x)$ to the ray in the direction $\theta$, we see that the definition of $f_0$ is analogous to the definition of $\varphi$ except that the supremum is taken over a smaller set. Hence, $f_0 \leq \varphi$.

On the other hand, if $\ell$ is a null line in the past of the graph of $u$, then the past of $\ell$ must also lie in the past of the graph of $u$. Since the past of $\ell$ is the same as the past of the unique null plane through $\ell$, this plane must also lie in the past of the graph of $u$. Applying this observation to the half-line $\{(r\theta_0,r - f_0(\theta_0))\in\R^{n,1}\,|\,r\geq 0\}$, we conclude that $u(\mathbf x) \geq \langle \theta_0, \mathbf x \rangle - f_0(\theta_0)$, and hence that $f_0 \geq \varphi$. This completes the proof of the theorem up to the final statement.

For the final statement, suppose that the graph of $u$ has mean curvature bounded below by $H$ and $\varphi = +\infty$ on an open set containing $\theta_0$. The linear isometry group $SO(n,1)$ acts on the sphere of null directions by conformal transformations, so up to the action of this group, we may assume that the open set contains an entire hemisphere centered at $\theta_0$. Then the domain of dependence of the graph of $u$ contains a spacelike ray $\{(r\theta_0,c)\in\R^{n,1}\,|\,r\geq 0\} = 0$ for some sufficiently large $c$. By Lemma \ref{lem: C0 bound}, the function $u$ is bounded above by $1/H$ along this ray, so $V(\theta_0) \leq 0$ and in particular $\theta_0 \notin L$. Since we have already seen that $L$ is a closed set containing $\{\theta \,|\, \varphi(\theta) < + \infty\}$, this completes the proof.
\end{proof}

The other commonly used notion of the asymptotics of an entire spacelike hypersurface $\Sigma$ is its asymptotic cut at future null infinity, which we now define. Introduce coordinates $\{t',t_a,\theta\}$ on the complement of the $x^{n+1}$ axis in $\R^{n,1}$ as follows: if $\{r > 0,\;\theta\in \mathbb S^{n-1}\}$ are spherical coordinates on $\R^n$, then set $t' = x^{n+1} - r$ and $t_a = x^{n+1} + r$. The function $t'$ is known as retarded time, and $t_a$ advanced time. Fixing $\theta$ defines a half-plane in $\R^{n,1}$, which meets the hypersurface $\Sigma$ along a spacelike curve. For $t_a$ large enough, this curve is the graph of a decreasing function $t' = f_\theta(t_a)$. The asymptotic cut at future null infinity of $\Sigma$ is defined to be the graph of the upper semicontinuous function $\psi(\theta) = \lim_{t_a \to +\infty} f_\theta(t_a): \mathbb S^{n-1} \to \R$. This definition becomes more geometrically intuitive if we identify the cylinder $\mathbb S^{n-1} \times \R$ with the component $\mathscr{I}^+$ of the boundary of the Penrose compactification of $\R^{n,1}$ as described in \cite[Section 5.1]{Hawking:1973aa}; then the closure of $\Sigma$ in the compactification meets $\mathscr{I}^+$ in the closure of the graph of $\psi$. 

We remark that this is a generalization of the traditional notion of a cut at future null infinity. Traditionally, a cut means the intersection of the closure of the null cone of a point in $\R^{n,1}$ with $\mathscr{I}^+$, which are just the graphs of affine functions on $\mathbb S^{n-1}$. In \cite{Stumbles:1981aa}, this is generalized to a ``BMS super-translated'' cut, meaning the graph of a sufficiently smooth function. According to the following proposition, our further generalization to upper semicontinuous functions is a very natural one:

\begin{prop} \label{prop: null cut} Let $\Sigma$ be an entire spacelike hypersurface in $\R^{n,1}$ with null support function $\varphi$. Then $\Sigma$ is asymptotic to the cut at future null infinity given by the graph of $-\varphi$.
\end{prop}

\begin{proof} Since $t' = x^{n+1}-r$, the definition of $\psi$ is the same as the definition of $f_0$ above up to a sign: $\psi = - f_0$. Hence the proposition follows immediately from the first part of Proposition \ref{prop: choi treibergs function}.
\end{proof}

\section{Uniqueness} \label{sec: uniqueness} 

In this section, we prove several comparison principles for entire hypersurfaces with bounds on their mean curvature. As a corollary, we obtain the uniqueness part of Theorem \ref{thm: existence}.

\begin{theorem} \label{thm: comparison}
Let $\Sigma_1$ and $\Sigma_2$ be entire spacelike hypersurfaces in $\R^{n,1}$. Suppose that $\Sigma_1$ has mean curvature bounded below by $H_1 > 0$, $\Sigma_2$ has constant mean curvature $H_2$, and $H_1 \geq H_2$. Suppose also that $\Sigma_2 \subset \mathcal{D}(\Sigma_1)$. Then $\Sigma_2 \subset J^+(\Sigma_1)$.
\end{theorem}

We recall that $J^+(\Sigma_1) = \Sigma_1 \cup I^+(\Sigma_1)$ is the causal future. The uniqueness of solutions in a regular domain is an immediate corollary:

\begin{cor*}[Uniqueness part of Theorem \ref{thm: existence}]
 For any regular domain $\mathcal{D}$, there is at most one entire hypersurface of constant mean curvature $H$ whose domain of dependence is $\mathcal{D}$.
\end{cor*}

The essential point of the proof of Theorem \ref{thm: comparison} is to apply the maximum principle to the distance between the hypersurfaces, but some care has to be taken because we don't have enough a priori control over the hypersurfaces at infinity. In particular, the containment $\Sigma_2 \subset \mathcal{D}(\Sigma_1)$, which implies $\mathcal{D}(\Sigma_2) \subset \mathcal{D}(\Sigma_1)$, does not a priori mean that $\Sigma_2$ is asymptotically in the future of $\Sigma_1$. 
However, it tells us the following:

\begin{lemma} \label{lem: minattained}
If $\Sigma$ is an entire spacelike hypersurface and $p$ is a point in $\mathcal{D}(\Sigma)$, then the square distance $\langle q - p, q - p \rangle$ attains its nonpositive minimum over $q \in \Sigma$.
\end{lemma}

\begin{proof}
By Proposition \ref{prop: D is open}, there is a compact set $K \subset \Sigma$ such that $p \in \mathcal{D}(K)$. Since any pair of points in $\Sigma$ is spacelike separated, $K = \Sigma \cap \mathcal{D}(K)$. Hence, for points $q$ in $\Sigma$ but outside $K$, there is no causal geodesic from $p$ to $q$, so the square distance from $p$ to $q$ is positive. Since $p \in \mathcal{D}(K)$, there is some point in $K$ which is connected to $p$ by a causal geodesic, so the square distance to $p$ is nonpositive. Since $K$ is compact and the square distance is continuous, it attains its nonpositive minimum over $\Sigma$ at some point of $K$.
\end{proof}

For two points $p$ and $q$ in $\R^{n,1}$, we will write $z = \langle q-p, q-p \rangle$, where we view $z$ as a function of $p$ and $q$. Recall (Definition \ref{def: Lorentzian distance}) that if $q \in J^+(p)$, the Lorentzian distance is defined by $d(p,q) = \sqrt{-z}$, and if $\Sigma$ is an achronal set with $\Sigma \cap J^+(p) \neq \emptyset$ then 
\[
d(p, \Sigma) = \sup_{q \in \Sigma \cap J^+(p)} d(p,q)~.
\]

Now we can state the second lemma we need in the proof of the comparison theorem.

\begin{lemma} \label{lem: C0 bound}
If $\Sigma$ is an entire spacelike hypersurface with mean curvature bounded below by a positive constant $H_0$ and $p \in \mathcal{D}(\Sigma) \cap I^-(\Sigma)$, then $d(p, \Sigma) < 1/H_0$.
\end{lemma}

\begin{proof} {This is a straightforward application of the maximum principle, which we describe in some detail, as we will build off of the computation in the proof of the next lemma.} By Lemma \ref{lem: minattained}, there is a point $q_0 \in \Sigma$ at which the square distance to $p$ attains its negative minimum, and thus the Lorentzian distance attains its positive maximum: $d(p, q_0) = d(p, \Sigma)$. We compute the intrinsic Laplacian on $\Sigma$ of the square distance to $p$ at the point $q_0$.

For a hypersurface embedded in $\R^{n,1}$ by a map $q$ with mean curvature $H(q)$ with respect to its unit normal $\nu$, the intrinsic Laplacian on the hypersurface of the embedding is given by
\begin{equation} \label{eqn: Delta q}
\begin{split}
\Delta q = n H(q) \nu.
\end{split}
\end{equation}
If we call the square distance to $p$, as a function on $\Sigma$, by $z(q) = \langle q - p, q - p \rangle$, then
\begin{equation} \label{eqn: Delta q z}
\begin{split}
\Delta z &= 2\langle \Delta q, q-p \rangle + 2 \langle \nabla_i q, \nabla^i q \rangle\\
& = 2nH(q) \langle \nu, q-p \rangle + 2 n.
\end{split}
\end{equation}
Here the term $\langle \nabla_i q, \nabla^i q \rangle$ is the pointwise Dirichlet energy of the embedding, written in Einstein notation. Since the embedding is isometric, the pointwise Dirichlet energy is equal to the rank, $n$.

Since the point $q_0$ is a critical point for $z$, the future normal vector $\nu$ at this point is parallel to $q_0-p$. More precisely, \[
q_0-p = d(p, q_0) \nu.
\]
Therefore, using that $\langle \nu, \nu \rangle = -1$, we get
\begin{equation} \label{eqn: Delta p rho}
\begin{split}
\Delta z (q_0)&=  2n(-H(q_0) d(p,q_0) + 1)\\
\end{split}
\end{equation}
Since $q_0$ is a minimizer for the square distance $z$, we must have $\Delta z (q_0) \geq 0$, and hence $d(p, \Sigma) = d(p,q_0) \leq 1/H(q_0)\leq  1/H_0$.

It remains only to rule out equality. By Proposition \ref{prop: D is open}, $\mathcal{D}(\Sigma)$ is open, so for $\epsilon$ small enough, the point $p - \epsilon(q_0 - p)$ is still in $\mathcal{D}(\Sigma)$. Running the same argument with $p$ replaced by $p - \epsilon(q_0 - p)$, we conclude that the inequality must be strict. 
\end{proof}

This bound has the following important consequence:

\begin{cor}\label{cor: dom of dep regular domain}
If $\Sigma$ is an entire hypersurface with mean curvature bounded below by a positive constant, then the domain of dependence of $\Sigma$ is a regular domain.
\end{cor}
\begin{proof}
By Proposition \ref{prop: entire past horizons}, the past horizon of $\Sigma$ is either empty, a single null hyperplane, or the past horizon of a regular domain. We show that in either of the first two cases, there would exist points $p \in \mathcal{D}(\Sigma) \cap I^-(\Sigma)$ such that $d(p,\Sigma)$ was arbitrarily large. For any point $q_0 \in \Sigma$, we have $d(p, \Sigma) \geq d(p, q_0)$. The level sets of the Lorentzian distance to $q_0$ are hyperboloids asymptotic to its past null cone; if the past horizon is empty or a single null hyperplane, each of these hyperboloids meets $\mathcal{D}(\Sigma)$, so we can make $d(p, q_0)$ arbitrarily large for $p \in \mathcal{D}(\Sigma)$. 

Hence, the past horizon $\mathcal{H}^-(\Sigma)$ is equal to the past horizon of some regular domain. We complete the proof by showing that future horizon of $\Sigma$ is empty. Otherwise, it would be an entire future horizon lying entirely in the future of $\mathcal{H}^-(\Sigma)$. Since $\mathcal{H}^-(\Sigma)$ is the past horizon of a regular domain, it has a spacelike support hyperplane. But by Proposition \ref{prop: entire past horizons} applied to future horizons, every nonempty entire future horizon is in the past of some null hyperplane. Clearly, no entire hypersurface can be sandwiched between a spacelike plane and a null plane. Hence, $\mathcal{H}^+(\Sigma)$ is empty and $\mathcal{D}(\Sigma) = I^+(\mathcal{H}^-(\Sigma))$, which is a regular domain.
\end{proof}

Now suppose that $\Sigma_1$ and $\Sigma_2$ are as in the statement of Theorem \ref{thm: comparison}; that is to say, $\Sigma_1$ and $\Sigma_2$ are entire spacelike hypersurfaces with $\Sigma_2 \subset \mathcal{D}(\Sigma_1)$, $\Sigma_1$ has mean curvature bounded below by $H_1$, $\Sigma_2$ has constant mean curvature $H_2$, and $H_1 \geq H_2 > 0$. Suppose for the sake of contradiction that $\Sigma_2$ meets the past of $\Sigma_1$. For $p \in \Sigma_2 \cap I^-(\Sigma_1)$, define
\[
u(p) = d(p, \Sigma_1)
\]
Since $\Sigma_2 \subset \mathcal{D}(\Sigma_1)$, we know by Lemma \ref{lem: C0 bound} that the function $u$ is bounded above by $1/H_1$.

\begin{lemma} \label{lem: subsolution}
 The inequality
\begin{equation} \label{eqn: laplacian of d}
\begin{split}
\Delta u \geq \frac{nH_1H_2u}{1-H_1 u} - \frac{n H_2 |\nabla u|}{1 - H_1 u} - \frac{|\nabla u|^2}{u}
\end{split}
\end{equation}
holds in the viscosity sense (i.e. $u$ is a viscosity subsolution).
\end{lemma}

We recall the definition of a viscosity subsolution. Let $a^{ij}(x, u, D u) D_{ij} u = F(x, u, D u)$ be an elliptic quasilinear differential equation, in the sense that $a^{ij}$ is a positive definite symmetric matrix. We say a function $\phi$ touches $u$ from above at a point $p$ if $\phi(p) = u(p)$ and $\phi \geq u$ in a neighborhood of $p$. 

\begin{defi} An upper semicontinuous function $u$ is a \emph{viscosity subsolution} of the equation $a^{ij}(x, u, D u) D_{ij} u = F(x, u, D u)$ if for any point $p_0$ and any $C^2$ function $\phi$ which touches $u$ from above at $p_0$, the inequality $a^{ij}(x, \phi, D \phi) D_{ij} \phi \geq F(x, \phi, D \phi)$ holds at the point $p_0$. It is a \emph{strict} viscosity subsolution if for any $\phi$ as above, strict inequality holds.
\end{defi}


\begin{proof}[Proof of Lemma \ref{lem: subsolution}]
Let $p_0$ be an arbitrary point in $\Sigma_2$ at which $u > 0$. By Lemma \ref{lem: minattained}, there is some point $q_0 \in \Sigma_1$ for which $d(p_0, q_0) = d(p_0, \Sigma_1)$ (Figure \ref{fig: p and q}). To estimate $\Delta u$ from below in the viscosity sense at $p_0$, we will find a smooth comparison subsolution $\underline{u}$ which touches $u$ from below at $p_0$ in the sense that $\underline{u} \leq u$ and $\underline{u}(p_0) = u(p_0)$. Let $q(p): \Sigma_2 \to \Sigma_1$, to be determined, be a smooth map with $q(p_0) = q_0$. In this way the function
\[
\underline{u}(p) = d(p, q(p))
\]
touches $u$ from below at $p_0$. Of course, it will be sufficient to define the germ of $q$ at $p_0$. Define also
\[
z = - \underline{u}^2 = \langle q-p, q-p \rangle~.
\]
\begin{figure}[!htb] 
\begin{center}
\begin{tikzpicture}
  \draw (.5,0) .. controls (2.5,-2) and (4.5,-2) .. (6.5,0);
  \draw (1,1) .. controls (3,-2) and (5,-3) .. (7,-1);
 \fill[black] (5,-2.05) circle (.05);
  \fill[black] (5.5,-.85) circle (.05);
  \node [above] (p0) at  (5,-2.05) {$p_0$};
    \node [above] (q0) at  (5.5,-.85) {$q_0$};
       \node [right] (S1) at  (6.5,0) {$\Sigma_1$};
           \node [right] (S2) at  (7,-1) {$\Sigma_2$};
  \end{tikzpicture}
\end{center}
\caption{\label{fig: p and q} The relative positions of $p_0$ and $q_0$.}
\end{figure}
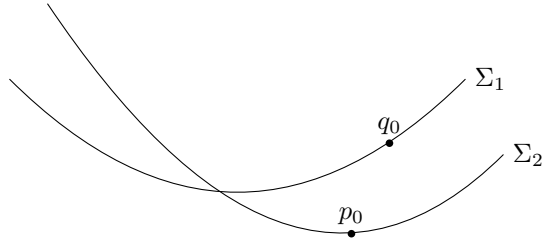

We first apply the chain rule to compute the Laplacian of $z$ at $p_0$. In Equations (\ref{eqn: chain rule 1}) and (\ref{eqn: chain rule}) below, the function $z$ on the left hand side should be interpreted as function on $\Sigma_2$, and on the right hand side as a function on $\Sigma_2 \times \Sigma_1$. The second partial derivatives should be interpreted as covariant derivatives or alternatively in normal coordinates $p^a$ on $\Sigma_2$ and $q^i$ on $\Sigma_1$ at the points $q_0\in\Sigma_1$ and $p_0\in\Sigma_2$.
\begin{equation} \label{eqn: chain rule 1}
\begin{split}
\frac{d z}{d p^a} = \frac{\partial z}{\partial p^a} + \frac{\partial z}{\partial q^i} \frac{\partial q^i}{\partial p^a}
\end{split}
\end{equation}
\begin{equation} \label{eqn: chain rule}
\begin{split}
\Delta z = \sum_a \left( \frac{\partial^2 z}{\partial p^a \partial p^a} + 2 \frac{\partial^2 z}{\partial p^a \partial q^i} \frac{\partial q^i}{\partial p^a} + \frac{\partial^2 z}{\partial q^i \partial q^j} \frac{\partial q^i}{\partial p^a} \frac{\partial q^j}{\partial p^a} + \frac{\partial z}{\partial q^i}\frac{\partial^2 q^i}{\partial p^a \partial p^a} \right)
\end{split}
\end{equation}

We now need to choose the function $q(p)$ to give a good upper bound for $\Delta z$. To motivate this choice, we begin with a couple of observations. First, the only term in Equation (\ref{eqn: chain rule}) which involves second derivatives of $q$ with respect to $p$ is the final term. Luckily, this term vanishes at $q_0$ because $q_0$ minimizes the square distance to $p_0$, and so $\frac{\partial z}{\partial q^i}(p_0,q_0) = 0.$ Therefore $\Delta z(p_0)$ depends only on the one-jet of the map $q(p)$. 

Second, by Equation (\ref{eqn: Delta q z}), the mean curvature is related to the intrinsic Laplacian of the square distance function. So, if our estimate for $\Delta z$ is to depend on the mean curvature $H_1$ as well as $H_2$, we had better have the operator $\frac{\partial^2}{\partial q^i \partial q^j} \frac{\partial q^i}{\partial p^a} \frac{\partial q^j}{\partial p^a}$ be a multiple of the Laplacian on $\Sigma_1$. In other words, we need the derivative $\frac{\partial q^i}{\partial p^a}$ to be an isometry up to scale.

Finally, given this constraint, we wish to minimize the cross term $2 \frac{\partial^2 z}{\partial p^a \partial q^i} \frac{\partial q^i}{\partial p^a}$, to give the best possible upper bound for $\Delta z$. 
\begin{figure}[!htb] 
\begin{center}
\begin{tikzpicture}
  \draw (-4,0) -- (4,0) -- (5.5,2) -- (-2.5,2) -- cycle;
  \draw (-4.25,-1) -- (4.25,1) -- (5.75,3) -- (-2.75,1) -- cycle;
  \draw (0,0) -- (1.5,2);
  \draw[thick, ->] (3.5,.1) .. controls (3.5,.4) .. (3.6,.7);
  \draw[thick, ->] (-3.3,-.1) .. controls (-3.3,-.4) .. (-3.4,-.7);
  \node[right] (TS2) at (4.5,.5) {$T_{p_0} \Sigma_2$};
  \node[right] (TS1) at (5.5,2.5) {$T_{q_0} \Sigma_1$};
\end{tikzpicture}
\end{center}
\caption{\label{fig: boost} An isometric boost between two spacelike hyperplanes.}
\end{figure}
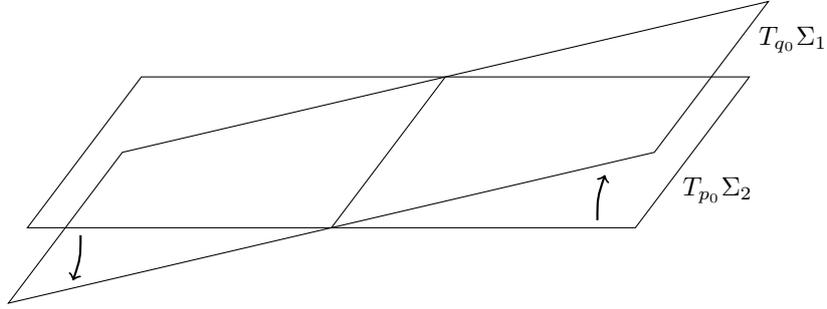
For a constant $\mu$ to be chosen in a minute, we choose the derivative $\frac{\partial q^i}{\partial q^a}$ to be the linear map $T_{p_0} \Sigma_2 \to T_{q_0}\Sigma_1$ which first isometrically boosts $T_{p_0} \Sigma_2$ onto $T_{q_0}\Sigma_1$ as in Figure \ref{fig: boost} and then scales by a factor of $\mu$. Since this map is in particular $\mu$ times an isometry, the third term in Equation (\ref{eqn: chain rule}) simplifies to
\[
\sum_a \frac{\partial^2 z}{\partial q^i \partial q^j} \frac{\partial q^i}{\partial p^a} \frac{\partial q^j}{\partial p^a} = \mu^2 \sum_i\frac{\partial^2 z}{\partial q^i \partial q^i}.
\]
Moreover we have a good estimate from above for the cross term. Indeed,
\begin{equation} \label{eqn: cross term}
\begin{split}
\sum_a 2 \frac{\partial^2 z}{\partial p^a \partial q^i} \frac{\partial q^i}{\partial p^a} &= \sum_a -4 \left\langle \frac{\partial}{\partial p^a}, \frac{\partial}{\partial q^i}\right\rangle \frac{\partial q^i}{\partial p^a}\\
&= \sum_a -4 \left\langle \frac{\partial}{\partial p^a}, \frac{\partial q^i}{\partial p^a}\frac{\partial}{\partial q^i}\right\rangle \\
&= - 4 \mu (n - 1 + |\langle \nu_1, \nu_2 \rangle|) \\
& \leq - 4 n \mu
\end{split}
\end{equation}
Here in the third equality, we have used that in $n-1$ directions the derivative of $q$ simply rescales by $\mu$, and in the final direction, the inner product between a unit vector and its image under the isometric boost of Figure \ref{fig: boost} is the same, up to sign, as the inner product of the normal vectors to the two planes.

Having chosen the one-jet of the map $q(p)$, we use Equation (\ref{eqn: Delta q z}) to write $\Delta z$ in terms of the mean curvatures $H_1$ and $H_2$. Namely, we have at the point $(p_0, q_0)$,
\begin{equation} 
\begin{split}
\sum_a\frac{\partial^2 z}{\partial p^a \partial p^a} &= -2n H_2 \langle \nu_2, q_0-p_0 \rangle + 2n \\
\sum_i\frac{\partial^2 z}{\partial q^i \partial q^i} &= 2nH(q_0)\langle \nu_1, q_0-p_0 \rangle + 2n  \\
\end{split}
\end{equation}
where $H(	q_0) \geq H_1$ is the mean curvature of $\Sigma_1$ at $q_0$. Keeping in mind that $q_0$ minimizes distance to $p_0$, so that
\[
q_0 - p_0 = d(p_0, q_0) \nu_1 = \underline{u}(p_0) \nu_1 = u(p_0) \nu_1,
\]
these become respectively
\begin{equation} \label{eqn: two Deltas and H}
\begin{split}
\sum_a\frac{\partial^2 z}{\partial p^a \partial p^a} & = 2n(H_2 u(p_0) |\langle \nu_1, \nu_2 \rangle| + 1) \\
\sum_i\frac{\partial^2 z}{\partial q^i \partial q^i} &\leq 2n(- H_1u(p_0) + 1) ~. \\
\end{split}
\end{equation}

The last ingredient we need is to express the term $|\langle \nu_1, \nu_2 \rangle|$ in terms of $|\nabla \underline{u}|$. Since $q_0$ is a minimizer of $d(p, q)$, the partial derivative with respect to $q$ vanishes, and so the total derivative at $p_0$ of $\underline{u} = d(p, q(p))$ is equal to its partial derivative with respect to $p$. This partial derivative is the projection onto $T_{p_0} \Sigma_2$ of the gradient of the distance as a function on $\R^{n,1}$, which is the vector $(q_0 - p_0)/d(p_0, q_0)$, which is just $\nu_1$. Writing the length of $\nu_1$ as the difference of its tangential and orthogonal components on $\Sigma_2$ gives $-1 = |\nabla \underline{u}|^2 - |\langle \nu_1, \nu_2 \rangle|^2$, in other words
\[
|\langle \nu_1, \nu_2 \rangle| = \sqrt{1 + |\nabla \underline{u}|^2}.
\]
We will just use the naive bound
\begin{equation} \label{eqn: gradient term}
\begin{split}
|\langle \nu_1, \nu_2 \rangle| \leq 1 + |\nabla \underline{u}|.
\end{split}
\end{equation}

Plugging (\ref{eqn: cross term}), (\ref{eqn: two Deltas and H}), and (\ref{eqn: gradient term})  into Equation (\ref{eqn: chain rule}) gives at the point $p_0$,
\begin{equation} \label{eqn: inequality with mu}
\begin{split}
\Delta z \leq 2n( H_2 u (1 + |\nabla \underline{u}|) + 1) - 4n\mu + 2n(-H_1 u + 1 )\mu^2~.
\end{split}
\end{equation}
We now choose $\mu$. By Lemma \ref{lem: C0 bound}, $-H_1 u + 1 > 0$, so the optimal choice is $\mu = \frac{1}{-H_1 u + 1}$, which after some algebra gives
\[
\Delta z \leq \frac{2n((H_2 - H_1)u + H_2 u|\nabla \underline{u}| - H_1H_2 u^2)}{1 - H_1 u}~.
\]
Finally, using $H_2 \leq H_1$ and $\underline{u}(p_0) = u(p_0) > 0$, together with the definition $z = -\underline{u}^2$, we arrive at
\[
\Delta \underline{u} \geq \frac{nH_1H_2u}{1-H_1 u} - \frac{n H_2 |\nabla \underline{u}|}{1 - H_1 u} - \frac{|\nabla \underline{u}|^2}{u}~.
\]
If $\phi$ is any smooth function that touches $u$ from above, then $\phi$ also touches $\underline u$ from above, and so $\nabla \phi(p_0) = \nabla \underline{u}(p_0)$ and $\Delta \phi(p_0) \geq \Delta \underline{u}(p_0)$. Hence Equation (\ref{eqn: laplacian of d}) holds in the viscosity sense.
\end{proof}

\begin{proof}[Proof of Theorem \ref{thm: comparison}]
We now complete the proof of the comparison theorem. Define $u$ as above on $\Sigma_2 \cap I^-(\Sigma_1)$, and extend it continuously by 0 to a function on all of $\Sigma_2$. The theorem is proved by showing that $u$ is nowhere positive. Let 
\[
F(u,\xi) = \frac{nH_1H_2u}{1-H_1 u} - \frac{n H_2 |\xi|}{1 - H_1 u} - \frac{|\xi|^2}{u}~.
\]
If $u$ attains a positive maximum at some point, then it is touched above by a constant function $m$ at that point. This contracts Lemma \ref{lem: subsolution} since $0 = \Delta m$ but $F(m,0) > 0$ if $m > 0$. To prove the general case, we will compare $u$ with a function of the form $\phi = m + \epsilon r^2$, where $r$ is the intrinsic distance to some point $p_0$ in $\Sigma_2$.

Suppose for the sake of contradiction that $\sup_{\Sigma_2} u = m > 0$. Following the argument of the Omori-Yau maximum principle (see \cite[Theorem 3]{Cheng:1975aa}), since $u$ is bounded above, for any $\epsilon$ we can find a point $p_0$ such that $u(p_0) > m - \epsilon$. Since $\Sigma_2$ has constant mean curvature, it is complete by Theorem \ref{thm: Cheng-Yau}, so the ball $B$ of radius one about $p_0$ is properly contained in $\Sigma_2$. Let $r$ be the intrinsic distance to $p_0$, and consider the function $u - \epsilon r^2$ on $B$. By construction, its value at $p_0$ is bigger than its supremum over the boundary of $B$, so it attains a maximum at some point $p_1$ in $B$. Let $\phi = u(p_1) -  \epsilon r(p_1)^2 + \epsilon r^2$, so that $\phi$ touches $u$ from above at $p_1$.


Since $\Sigma_2$ has constant mean curvature, the theorem of Cheng and Yau tells us that $\phi$ is smooth: $\Sigma_2$  has non-positive sectional curvature (Proposition \ref{prop: nonpos sec curv}), so it has no conjugate points, so $r^2$ and $\phi$ are smooth functions. It remains to establish that for $\epsilon$ small enough, $\phi$ is a strict supersolution of (\ref{eqn: laplacian of d}).

Recall that by Proposition \ref{prop: Ricci lower bound}, the Ricci curvature of $\Sigma_2$ is bounded below by $-n^2H_2^2/4$ times the metric. Hence by the gradient comparison theorem (\cite[Lemma 1]{MR431040}), there is a constant $C(n,H_2)$ such that on the ball $B$ we have $\Delta r \leq C(n,H_2)/r$ and so $\Delta \phi \leq 2 \epsilon C(n,H_2)$. We also have on $B$ that $|\nabla \phi| = 2 r \epsilon \leq 2 \epsilon$. 

Now let $\delta =  \frac{m}{2} \leq \frac{1}{2H_1}$, and choose $\epsilon$ small enough that
\begin{enumerate}
\myitem{$i)$}\label{omyau item 1} $\epsilon < \delta$, which implies $u(p_1) \geq u(p_0) > \delta$;
\myitem{$ii)$}\label{omyau item 2} $\epsilon < \mathrm{min}\left(\frac{H_1\delta}{6}, \sqrt{\frac{nH_1H_2\delta^2}{12(1-H_1\delta)}}\right)$, which together with \ref{omyau item 1} and the fact that $|\nabla \phi| \leq 2 \epsilon$ implies \[F(\phi(p_1), \nabla \phi(p_1)) > \frac{nH_1H_2\delta}{3(1-H_1\delta)};\]
\myitem{$iii )$}\label{omyau item 3} $\epsilon < \frac{1}{2C(n,H_2)} \cdot \frac{nH_1H_2\delta}{3(1-H_1\delta)}$, which together with \ref{omyau item 1} and \ref{omyau item 2} and the fact that $\Delta \phi \leq 2\epsilon C(n,H_2)$ implies \[\Delta \phi(p_1) < F(\phi(p_1), \nabla \phi(p_1)).\]
\end{enumerate}

Since $u(p_1) > 0$, it is a viscosity subsolution to the equation $\Delta u = F(u, \nabla u)$ at $p_1$ by Lemma \ref{lem: subsolution}, but since $\phi$ touches $u$ from above at $p_1$, this gives a contradiction. Hence $u$ cannot be positive, and $\Sigma_2 \subset J^+(\Sigma_1)$.
\end{proof}

We have completed the proof of the Comparison Theorem \ref{thm: comparison}, and hence also of the corollary that any two hypersurfaces of the same constant mean curvature sharing the same domain of dependence must coincide. It also follows that two hypersurfaces with different constant mean curvatures sharing the same domain of dependence are time-ordered by the inverse of their mean curvatures. These are the only two consequences of the comparison theorem that we will need in the remainder of this paper, and hence Theorem \ref{thm: comparison} is sufficient for our purposes. However, a stronger statement of the comparison theorem is possible, and in the remainder of this section we will sketch this argument.

Recall that in Theorem \ref{thm: comparison}, we assumed that $\Sigma_2$ had constant mean curvature $H_2 > 0$, but only that $\Sigma_1$ had mean curvature bounded below by $H_1 \geq H_2$. In the proof, we considered the distance to $\Sigma_1$ as a function on $\Sigma_2$. If instead, we consider the distance to $\Sigma_2$ as a function on $\Sigma_1$, we can prove the following:

\begin{theorem} \label{thm: other comparison}
Let $\Sigma_1$ and $\Sigma_2$ be entire spacelike hypersurfaces in $\R^{n,1}$. Suppose that $\Sigma_1$ has constant mean curvature $H_1 > 0$, $\Sigma_2$ has mean curvature bounded above by $H_2$, and $H_1 \geq H_2$. Suppose also that $\Sigma_2 \subset \mathcal{D}(\Sigma_1)$. Then $\Sigma_1 \subset J^-(\Sigma_2)$.
\end{theorem}
We remark that we do not need to assume $H_2 \geq 0$.

\begin{proof}

Suppose $\Sigma_1 \cap I^+(\Sigma_2)$ were nonempty. For $q \in \Sigma_1 \cap I^+(\Sigma_2)$, let 
\[
u(q) = d(\Sigma_2, q).
\]
The analog of Lemma \ref{lem: subsolution} in this case is that in the viscosity sense,
\begin{equation} \label{eqn: laplacian of d 2}
\begin{split}
\Delta u \geq \frac{nH_1H_2 u}{H_2 u + 1} - \frac{|\nabla u|^2}{u}
\end{split}
\end{equation}
where $\Delta$ is the Laplacian on $\Sigma_1$. The proof parallels the proof of Lemma \ref{lem: subsolution}, except that instead of the bound $|\langle \nu_1, \nu_2 \rangle | \leq 1 + |\nabla u|$ in Equation \ref{eqn: gradient term}, we use the even simpler bound $|\langle \nu_1, \nu_2 \rangle | \geq 1$. We now spell this out in a little more detail. For any point $q_0$ in $\Sigma_1 \cap I^+(\Sigma_2)$, we can find a point $p_0 \in \Sigma_2$ maximizing the distance: indeed, since $\Sigma_2 \subset \mathcal{D}(\Sigma_1)$ which is a regular domain by Corollary \ref{cor: dom of dep regular domain}, the hypersurface $\Sigma_2$ can have no future horizon, so $I^+(\Sigma_2) \subset \mathcal{D}(\Sigma_2)$; since $q_0 \in I^+(\Sigma_2) \subset \mathcal{D}(\Sigma_2)$ such a point $p_0$ exists by Lemma \ref{lem: minattained}. 

To establish the bound (\ref{eqn: laplacian of d 2}), define $p(q): \Sigma_1 \to \Sigma_2$ so that $p(q_0) = p_0$ and $\partial p / \partial q$ is the inverse of the boost in Figure \ref{fig: boost} followed by dilation by $\mu$, and define $\underline{u}(q) = d(p(q), q)$ and $z = - \underline{u}^2 = \langle q - p, q - p \rangle$. Then just as in Lemma \ref{lem: subsolution} we find
\begin{equation}
\begin{split}
\Delta z \leq 2n(- H_1 u + 1) - 4n\mu + 2n(H_2 u + 1)\mu^2
\end{split}
\end{equation}
where now we consider $z$ as a function on $\Sigma_1$. Taking $\mu = \frac{1}{H_2 u + 1}$, using $H_1 \geq H_2$, and rewriting in terms of $u$, gives the Laplacian bound (\ref{eqn: laplacian of d 2}). 

The remainder of the proof proceeds exactly as for Theorem \ref{thm: comparison}, using the completeness of $\Sigma_1$ that follows from its constant mean curvature by the theorem of Cheng and Yau. Note that we still have the upper bound $u \leq \frac{1}{H_1}$ by Lemma \ref{lem: C0 bound} because $p_0 \in \mathcal{D}(\Sigma_1)$.
\end{proof} 

For the most general version of the comparison theorem, we combine Theorems \ref{thm: comparison} and \ref{thm: other comparison} with the existence part of Theorem \ref{thm: existence} proved in Section \ref{sec: existence}. We stress that we do not need the following theorem in the proof of existence.

\begin{theorem} \label{thm: general comparison}
 Let $\Sigma_1$ and $\Sigma_2$ be entire spacelike hypersurfaces in $\R^{n,1}$. Suppose that $\Sigma_1$ has mean curvature bounded below by $H_1 > 0$, $\Sigma_2$ has mean curvature bounded above by $H_2$, and $H_1 \geq H_2$. Suppose also that $\Sigma_2 \subset \mathcal{D}(\Sigma_1)$. Then $\Sigma_1 \subset J^-(\Sigma_2)$.
\end{theorem}

\begin{proof}
By Corollary \ref{cor: dom of dep regular domain}, the domain of dependence $\mathcal{D}(\Sigma_1)$ is a regular domain. Using Theorem \ref{thm: existence}, let $\Sigma$ be the unique hypersurface of constant mean curvature $H$ with domain of dependence $\mathcal{D}(\Sigma_1)$ for any fixed value of $H$ with $H_1 \geq H \geq H_2$. By Theorem \ref{thm: comparison}, since $\Sigma \subset \mathcal{D}(\Sigma_1)$, the hypersurface $\Sigma$ lies weakly in the future of $\Sigma_1$. By Theorem \ref{thm: other comparison}, since $\Sigma_2 \subset \mathcal{D}(\Sigma_1) = \mathcal{D}(\Sigma)$, the hypersurface $\Sigma$ lies weakly in the past of $\Sigma_2$. Therefore, $\Sigma_1$ lies weakly in the past of $\Sigma_2$.
\end{proof}

\section{Existence} \label{sec: existence}
In this section we prove the following theorem about the existence of a CMC hypersurface in any regular domain. Our proof relies on a combination of the techniques of \cite{choitreibergs} and \cite{MR2904912}.


\begin{theorem*}[Existence part of Theorem \ref{thm: existence}]
For any regular domain $\mathcal{D}$ and any $H>0$ there exists an entire spacelike hypersurface $\Sigma$ with  constant mean curvature $H$ and domain of dependence $\mathcal D$. 
\end{theorem*}

This result is a little generalization of  Theorem 6.2 of \cite{choitreibergs}, and the strategy is essentially the same as in that paper. 
The new ingredient with respect to \cite{choitreibergs} is the observation of \cite{MR2904912} that level sets of cosmological time can be used as good barriers in any regular domain. The proof is then a direct application of Proposition 6.1 of \cite{choitreibergs}.
For completeness in Subsection \ref{ss:ctargument} we will give a short  overview of the argument of Choi and Treibergs.

\subsection{Barrier argument}

Let $\Sigma_u=\{(\mathbf x,u(\mathbf x))\in\R^{n,1}\,|\,\mathbf x\in\Omega\}$ be a spacelike hypersurface, for $u:\Omega\to\R$ a differentiable function on a domain $\Omega$ with $|Du|<1$. We call such $u$ a spacelike function.
If $u$ is smooth, the mean curvature of 
$\Sigma_u$ at a point $(\mathbf{x}, u(\mathbf{x}))$ is given then by the expression
\[
   \frac{1}{n}\left(\frac{1}{\sqrt{1-|Du|^2}}\sum_{i}D_{ii}u+\frac{1}{(1-|Du|^2)^{3/2}}\sum_{ij} D_iuD_ju D_{ij}u\right)(\mathbf x)\,.
\]

So spacelike graphs of constant mean curvature $H$ bijectively correspond to solutions of the problem
\begin{equation}\tag{CMC}\label{eq:cmcproblem}
\begin{cases}
L_{H}(u)=(1-|Du|^2)\sum_{i}D_{ii}u+\sum_{ij} D_iuD_ju D_{ij}u-nH(1-|Du|^2)^{3/2}=0\,,\\
|Du|<1\,.
\end{cases}
\end{equation}

{The operator $L$ is quasi-linear and elliptic on the domain of spacelike functions. In the following we will normalize $H=1$ and we will simply denote $L_1$ by $L$.}




\begin{lemma}\label{lm:barr}
Let $\mathcal{D}$ be a regular domain and $T:\mathcal{D}\to(0,+\infty)$ the cosmological time. Let $v_0, v_1:\mathbb R^n\to\mathbb R$ be the functions whose graphs coincide with
$\partial \mathcal{D}$ and $S_1=T^{-1}(1)$. Then $v_0$ and $v_1$ are respectively a  sub- and super-solution of \eqref{eq:cmcproblem} in the viscosity sense.
\end{lemma}
 \begin{proof}
 
 Since $\partial\mathcal{D}$ is a past horizon, for any $\mathbf{x}\in\mathbb R^n$ there is $\mathbf{\xi}\in\mathbb R^n$ with $|\mathbf{\xi}|=1$ such that $v_0(\mathbf{x}+t\mathbf{\xi})=v_0(\mathbf{x})+t$ for small $t$ (in fact, since it is entire, this holds for all $t \geq 0$ by Proposition \ref{prop: entire past horizons}). 
  This implies that if $v$ is any spacelike function, then the function $v-v_0$ cannot have local interior minimum. In fact we have that $(v-v_0)(\mathbf{x}+t\mathbf{\xi})<(v-v_0)(\mathbf{x})$.
 Thus no spacelike function $v$ can touch $v_0$ from above.
  
  In order to prove that $v_1$ is a supersolution in the viscosity sense  we will show that for any $\mathbf{x}\in\mathbb R^n$ there exists a function $v^{\mathbf{x}}$, solution of $L$, touching from above $v_1$ at $\mathbf{x}$.
  In fact   we claim that for any point $p=(\mathbf x,u_1(\mathbf{x}))\in S_1$ there is a hyperboloid $\mathcal H(\mathbf{x})$ passing through $p$ and lying  above $S_1$.
  The function which defines $\mathcal H(\mathbf x)$, say $v_1^{\mathbf  x}$, is a solution of $L$ 
   which satisfies $v_1^{\mathbf x}\geq v_1$ and $v_1^{\mathbf x}(\mathbf{x})=v_1(\mathbf{x})$.
  
 In order to prove the claim let $r$ be the point on $\partial\mathcal D$ such that $\langle p-r, p-r\rangle=T(p)=1$, 
 and consider the hyperboloid $\mathcal H(\mathbf{x})=r+\mathbb H^n$. Clearly $p\in \mathcal H(\mathbf{x})\subset\mathcal D$.
 Moreover for any $p'\in \mathcal  H(\mathbf{x})$ we have that $\langle p'-r, p'-r\rangle=1$ so the cosmological time of points in $\mathcal H(\mathbf{x})$ cannot be less than $1$, that is $\mathcal H(\mathbf{x})\subset {J^+(S_1)}$.
 \end{proof}

As the coefficients of the quasi-linear elliptic operator $L$ depend only on the derivatives of $u$, the Comparison Principle applies (see \cite[Theorem 10.1]{GT}): if $u,v$ are twice differentiable spacelike functions on a compact domain $\Omega\subset\mathbb R^n$,
 such that $L(u)<L(v)$ on $\Omega$ and $u-v$ is nonnegative on $\partial\Omega$,  then $u-v$ is nonnegative on the whole domain $\Omega$. We remark here that a comparison principle holds if $v$ is only a viscosity subsolution:
\begin{lemma} \label{lemma comparison viscosity}
Let $v:\Omega\to\mathbb R$ be a viscosity  subsolution of \eqref{eq:cmcproblem}. For any solution $u$ defined on $\Omega'\Subset \Omega$, if $v\leq u$ on $\partial\Omega'$ then $v\leq u$ on $\Omega'$.
\end{lemma}
\begin{proof}
 Let $u$ be a solution of \eqref{eq:cmcproblem} on $\Omega'$ with $v\leq u$ on $\partial\Omega'$. If $v>u$ at some interior point then, replacing $u$ by $u+c$ for $c>0$, one can arrange that $u$ touches $v$ from above at some point. If $v$ is a strict viscosity subsolution, this would imply that $L(u)>0$ and gives a contradiction. If $v$ is only a viscosity subsolution, one still gets the same conclusion by a perturbation argument.
\end{proof}
 Clearly the analog  discussion holds for supersolutions, by reversing inequalites.
The proof of Theorem \ref{thm: existence} is then consequence of the following  proposition stated in \cite{choitreibergs}.
\begin{prop}[{\cite[Proposition 6.1]{choitreibergs}}]\label{pr:extec}
Suppose there exist functions $\underline v\leq \overline v \in C^{0,1}(\mathbb R^n)$ which are {convex and proper} {viscosity} sub- and super-solutions to the constant mean curvature equation \eqref{eq:cmcproblem}. Then there exists a smooth solution $u$ of \eqref{eq:cmcproblem}
whose graph is an entire spacelike hypersurface of constant mean curvature $1$ in $\mathbb R^{n,1}$, which satisfies 
\[
     \underline v(\mathbf{x})\leq u(\mathbf{x})\leq \overline v(\mathbf{x})\quad\textrm{for all } \mathbf{x}\in\mathbb R^n\,.
\]
\end{prop}

In Section \ref{ss:ctargument} we will provide an outline of the proof of Proposition \ref{pr:extec}. Assuming this proposition, let us now prove the existence part of Theorem \ref{thm: existence}. {The obstacle to directly applying Proposition \ref{pr:extec} is the properness assumption for the barriers.}

\begin{proof}[Proof of the existence part of Theorem \ref{thm: existence}]
Up to scaling, we can assume $H=1$. 
By Lemma \ref{lm:barr} 
{the convex functions $v_0$ and $v_1$ whose graphs are $\partial\mathcal{D}$ and $S_1$ are  a viscosity sub-solution and  super-solution of \eqref{eq:cmcproblem} respectively.}

Let us consider first the case where the set of  $\xi\in\Hyp^n$ which are orthogonal to some support plane of $\partial\mathcal{D}$  has non-empty interior. 
Then up to applying an isometry of Minkowski space we can assume that the vector $e=(0,0,\ldots,0,1)$ lies in the interior of this set.

{By assumption, there exists $\epsilon>0$ such that for every $\xi$ in the closed geodesic ball $B$ centered at $e$ of radius $\epsilon$, $\mathcal D$ admits a spacelike support hyperplane orthogonal to $\xi$. Since the points $\xi\in\partial B$ can be written as $\xi=(\cosh\epsilon)e+(\sinh\epsilon)\theta$ for $\theta\in \mathbb S^{n-1}$, such support plane is of the form $P_{\theta,c}=\{x_{n+1}=\tanh\epsilon\langle\theta,\mathbf x\rangle-c\}$ for some constant $c=c(\theta)$. By continuity of the support function in $B$, we can indeed find a constant $c$ independent of $\theta$ such that $\partial \mathcal D$ is in the future of $P_{\theta,c}$ for all $\theta\in \mathbb S^{n-1}$. 
Now, for every $\mathbf x\neq 0$ we can pick $\theta=\mathbf x/\|\mathbf x\|$ and deduce that $v_1(\mathbf x)>v_0(\mathbf x)\geq\mathrm{tanh}\epsilon \|\mathbf x\|-c$, thus showing that $v_0$ and $v_1$ are proper functions.}
 A direct application  of Proposition \ref{pr:extec} with $\underline v=v_0$ and $\overline v=v_1$, shows the existence of an entire  spacelike hypersurface $\Sigma$ of constant mean curvature equal to $1$ contained between $\partial\mathcal D$ and
 $S_1$. Since $\mathcal{D}(S_1) \subset \mathcal{D}(\Sigma) \subset \mathcal{D}$ and $\mathcal{D}(S_1) = \mathcal{D}$ by Proposition \ref{prop:T level sets}, we see that $\mathcal{D}(\Sigma) = \mathcal{D}$.

Assume now that support directions of $\mathcal D$ are contained in a hyperplane $P$ of $\mathbb H^n$. The hyperplane $P$ is the intersection of $\mathbb H^n$ with a timelike hyperplane  which we can identify to 
$\mathbb R^{n-1,1}$. In this way $\mathbb R^{n,1}$ is identified to  $\mathbb R\times\mathbb R^{n-1,1}$. Now by the assumption on the support directions we have that $\mathcal D$ splits as $\R\times \mathcal D^0$,
where $\mathcal D^0$ is a regular domain of $\mathbb R^{n-1,1}$. In fact we have that $\partial\mathcal D=\mathbb R\times\partial\mathcal D^0$ and $S_1=\mathbb R\times (S^0)_1$, where $(S^0)_1$ denotes the level set of the cosmological time of $\mathcal D^0$. Now by an inductive argument on the dimension we can assume that there is an entire hypersurface $\Sigma^0$ in $\mathbb R^{n-1,1}$  of constant mean curvature equal to $\frac{n}{n-1}$ and contained in $\mathcal D^0\cap J^{-}((S^0)_1)$.
The hypersurface $\Sigma=\mathbb R\times\Sigma^0$ has constant mean curvature equal to  $1$ and is contained in $\mathcal D\cap J^-(S_1)$. As in the previous case we can then conclude that $\mathcal D(\Sigma)=\mathcal D$.
\end{proof}

\subsection{Outline of the proof of Proposition \ref{pr:extec}}\label{ss:ctargument}
Proposition \ref{pr:extec} is stated in \cite{choitreibergs}, while its proof is referred to \cite{Treibergs:1982aa,MR680653}. In fact in the referred papers there are all the steps to prove that proposition, however it is never stated in the form we need.
So in order to help the reader we give here a brief overview of the proof of Proposition \ref{pr:extec}.

The first step is the existence of solution of the Dirichlet problem, which we do not prove.

\begin{prop}[{\cite[Proposition 6]{Treibergs:1982aa}}]\label{pr:dirichlet}
Let $\Omega$ be a compact convex domain in $\mathbb R^n$ with $C^{2,\alpha}$-boundary, and $k\in\mathbb R$.
There exists a function $u\in C^{2,\alpha}(\bar{\Omega})$ solving the Dirichlet problem
\begin{equation}
\begin{array}{l}
L(u)=(1-|Du|^2)\sum_{i}D_{ii}u+\sum_{ij} D_iuD_ju D_{ij}u-n(1-|Du|^2)^{3/2}=0\,,\\
u=k \quad\text{ on }\partial\Omega\,,
\end{array}
\end{equation}
such that
\[
   |Du|<c(n, H, \mathrm{diam}\Omega)<1~.
\]
\end{prop}
In fact $u$ is smooth by standard elliptic regularity theory. Now for any $k$ we denote by $\Omega_{+}(k)\subset\Omega_{-}(k)$  the $k$-sublevel  sets of $\overline v$ and $\underline v$ respectively. They are compact and convex by the assumption on $\underline v$ and $\overline v$.
Let $\Omega(k)$ be a convex subset with $C^{2,\alpha}$ boundary such that $\Omega_{+}(k)\subset\Omega(k)\subset\Omega_{-}(k)$.
Proposition \ref{pr:dirichlet} implies the existence of a function $u_k:\Omega(k)\to\mathbb R$ solving $L(u_k)=0$, $|Du_k|<1$ and $u_k|_{\partial\Omega(k)}=k$.
Moreover we have $\underline v\leq u_k\leq \overline v$ over $\Omega(k)$ {by Lemma \ref{lemma comparison viscosity}.} 
By construction, the family of domains $(\Omega(k))_{k\in\mathbb N}$ is an exhaustion of $\mathbb R^n$.
Since the $u_k$ are $1$-Lipschitz functions, up to taking a subsequence we can suppose that $u_k\rightarrow u_\infty$ uniformly on compact sets, with $\underline v\leq u_\infty\leq \overline v$.
We will prove that the convergence is in fact smooth on compact subsets.

Let $K$ be a compact region of $\mathbb R^n$. As in Step 6 of \cite[Theorem 1]{Treibergs:1982aa}, let us prove that the $C^3$-norm of $u_k$ is uniformly bounded over $K$.
Take a point $\mathbf{x}_0\in\mathbb R^n\setminus K$. In Step 5 of \cite[Theorem 1]{Treibergs:1982aa} it is shown that  there are constant $r_1<r_2$ such that for $k$ sufficiently large we have that
\[
\begin{array}{l}
d_k(\mathbf{x}_0, \mathbf x)<r_1\quad\forall\mathbf{x}\in K\,,\\
d_k(\mathbf{x}_0, \partial\Omega(k))>r_2\,,
\end{array}
\]
where $d_k(\mathbf{x},\mathbf{y})$ denotes here the intrinsic distance on the graph $\Sigma_{u_k}$ between $(\mathbf{x}, u_k(\mathbf{x}))$ and $(\mathbf{y}, u_k(\mathbf{y}))$.
Fix $a\in (r_1, r_2)$. We can then apply \cite[Proposition 3, Proposition 4]{Treibergs:1982aa} and prove a uniform bound of both the norm and the first covariant derivatives of the second fundamental form  of the hypersurface $\Sigma_{u_k}$ over $K$.
The bound on the norm of the second fundamental form implies that the Gauss maps of $\Sigma_{u_k}$ are uniformly Lipschitz on $K$. 
This implies that either the hypersurfaces $\Sigma_{u_k}$ are uniformly spacelike or the restriction of $\Sigma_{u_\infty}$ over $K$ is a portion of a lightlike plane.
On the other hand, since $u_\infty\geq \underline v$, we have that $\Sigma_{u_\infty}$ is not a lightlike plane, so taking $K$ sufficiently big the latter case cannot hold.
In conclusion there is a constant $c<1$ independent of $k$  such that $|D u_k|(\mathbf x)<c$ for every $\mathbf{x}\in K$ and $k\in\mathbb N$.

Now fix the standard frame $e_1,\ldots, e_n$ on $\mathbb R^n$ and denote by $\{\tilde e_i=e_i+ du_k(e_i)e_{n+1}\}$ the induced frame of $\Sigma_{u_k}$.
Notice that over $K$, $g_{ij}=\langle \tilde e_i, \tilde e_j\rangle$ is a uniformly bounded positive symmetric matrix, in the sense that its eigenvalues are bounded away from $0$ and $\infty$.
Putting $h_{ij}=\II(\tilde e_i, \tilde e_j)$ we have that $|\II|^2=\tr (g^{-1}h^2g^{-1})$, so that $\sum h_{ij}^2$ is also uniformly bounded over $K$.
On the other hand we have 
\begin{equation}\label{eq:ss}
h_{ij}=\frac{1}{\sqrt{1-|Du_k|^2}}D^2_{ij}u_k+\frac{1}{(1-|Du_k|^2)^{3/2}}f_{ij}(Du_k)
\end{equation}
 where $f_{ij}:\mathbb R^n\to\mathbb R$ are given smooth functions independent of $k$. Thus we deduce that the Hessian of $u_k$ is uniformly bounded over $K$.

This implies that Christoffel symbols of the Levi Civita connection of $\Sigma_{u_k}$ are uniformly bounded over $K$.
On the other hand, we have
\[
   \nabla_\ell h_{ij}=\partial_{\ell}h_{ij}-\Gamma_{i\ell}^{s}h_{is}-\Gamma_{j\ell}^{s}h_{js}\,.
\]
So {the bound on the first covariant derivative of the second fundamental form gives} that the derivatives of $h_{ij}$ are also uniformly bounded on $K$.
Differentiating \eqref{eq:ss} and using that $\|u_k\|_{C^2}$ is uniformly bounded on $K$ we deduce an estimate for the third derivatives of $u_k$.
So we conclude that $\|u_k\|_{C^3}$ is uniformly bounded and $|Du_k|<c<1$ over $K$.

By applying now a standard bootstrap argument we conclude that all the derivatives of $u_k$ are bounded over $K$, and by the Ascoli-Arzel\`a Theorem we conclude that $u_k$ lies in a compact subset of $C^{\infty}(K)$. Up to taking a subsequence, we can assume that $u_k$ converges $C^\infty$ to a limit $u_\infty$, which is a solution to \eqref{eq:cmcproblem} and satisfies $\underline v\leq u_\infty\leq \overline v$.

\section{CMC foliations} \label{sec: foliation}
In this section we prove that the entire CMC hypersurfaces foliate regular domains. That is, we will prove:

\begin{theorem*}[Foliation part of Theorem \ref{thm: existence}]
Let $\mathcal D$ be any regular domain. As $H$ varies in $(0,+\infty)$, the entire hypersurfaces with domain of dependence $\mathcal D$  and constant mean curvature $H$ analytically foliate $\mathcal D$.
\end{theorem*}

We shall first show that the CMC hypersurfaces continuously foliate $\mathcal D$, and then the analytic dependence on $H$. The proof will be split in Sections \ref{subsec: cont foliation} and \ref{subsec: smooth foliation}.

\subsection{Continuous foliations} \label{subsec: cont foliation}

To show that each regular domain $\mathcal D$ is foliated by CMC hypersurfaces,
let us denote by $\Sigma_H$ the unique entire hypersurface of CMC $H>0$ such that $\mathcal D(\Sigma_H)=\mathcal D$.

\begin{lemma} \label{lemma cont fol1}
The hypersurfaces $\Sigma_H$ are pairwise disjoint.
\end{lemma}
\begin{proof} In fact, by Theorem \ref{thm: comparison}, if $H_1>H_2$ then $\Sigma_2$ does not meet the past of $\Sigma_1$. By the strong maximum principle, $\Sigma_2$ lies strictly in the future of $\Sigma_1$.
\end{proof} 
It thus remains to show that any point of $\mathcal D$ is in some hypersurface $\Sigma_H$.

\begin{lemma} \label{lemma cont fol2}
Given every $p\in\mathcal D$, there exists $H$ such that $p\in\Sigma_H$.
\end{lemma}
\begin{proof}
Again from Theorem \ref{thm: comparison}, if $u_H$ is the function on $\R^n$ defining  $\Sigma_H$, $u_{H_1}<u_{H_2}$ if $H_1>H_2$.
Let $p\in \mathcal D$, and let 
\begin{equation}\label{eq:defi Hpm}
H_-:=\inf\{H\,:\,p\in I^+(\Sigma_H)\}\qquad \text{and}\qquad H_+:=\sup\{H\,:\,p\in I^-(\Sigma_H)\}~.
\end{equation}
We first claim that both sets of which we take the infimum/supremum in \eqref{eq:defi Hpm} are non-empty, so that $H_-,H_+\in(0,+\infty)$. Since the point $p$ belongs to at most one $\Sigma_H$ by disjointness (Lemma \ref{lemma cont fol1}), 
it will then follow that $H_-=H_+$.

For the claim, first let $k>0$ be such that $p$ is in the future of $S_k$, where $S_k=T^{-1}(k)$ are the level sets of the cosmological time. By Lemma \ref{lem: C0 bound}, for large $\bar H$ the CMC hypersurface $\Sigma_{\bar H}$ is in the past of $S_k$, hence $p\in I^+(\Sigma_{\bar H})$. Concerning $H_+$, we shall show that $p$ is in the past of some $\Sigma_{\bar H}$. For this purpose, let $\mathcal D_0$ be a wedge as in Figure \ref{fig:trough} --- that is a regular domain which is the intersection of exactly two future half-spaces bounded by non-parallel lightlike hyperplanes --- such that $\mathcal D\subset \mathcal D_0$. Such a $\mathcal D_0$ exists because by definition of regular domain $\mathcal D$ is the intersection of at least two future half-spaces bounded by non-parallel lightlike hyperplanes. Since the entire CMC hypersurfaces with domain of dependence $\mathcal D_0$ (which are troughs) foliate $\mathcal D_0$, there exists $\bar H$ such that $p$ is in the past of the trough having constant mean curvature $\bar H$. Then by Theorem \ref{thm: comparison}, $p$ is also in the past of $\Sigma_{\bar H}$, which concludes the claim. 

Now let $H_0:=H_-=H_+$. Define
$$u_-:=\sup_{H>H_0} u_H\qquad \text{and}\qquad u_+:=\inf_{H<H_0} u_H~.$$ 
Then as $H\to H_0^-$, $u_H$ converges uniformly on compact sets to $u_-$, and by an argument similar to that given in Section \ref{ss:ctargument}, the convergence is indeed smooth on compact sets. (One uses again the function defining $\partial\mathcal D$ as a subsolution to ensure that the $\Sigma_H$ are uniformly spacelike once the uniform boundedness of the second fundamental form is established.)
Hence $u_-$ defines an entire hypersurface of constant mean curvature $H_0$ with domain of dependence $\mathcal D$ . The same holds for $u_+$. By uniqueness of Theorem \ref{thm: uniqueness}, $u_-=u_+$ and thus $p\in \Sigma_{H_0}$.
\end{proof}
Lemma \ref{lemma cont fol1} and \ref{lemma cont fol2} together show that every regular domain $\mathcal D$ is continuously foliated by the CMC hypersurfaces with domain of dependence $\mathcal D$.

\subsection{Analyticity of the foliation} \label{subsec: smooth foliation}

From standard elliptic regularity theory {\cite[Theorem 6.17]{GT} \cite{Hopf:1932aa}}, we know that every CMC hypersurface is analytic in the sense that it is locally the graph of an analytic function on $\R^n$. In particular, the foliation of the previous section has analytic leaves. In this section, we show that the foliation itself is analytic in the sense that it can be trivialized by analytic charts. 

Fix any regular domain $\mathcal{D}$ and $H > 0$, and let $\Sigma$ be the leaf of the CMC foliation of $\mathcal{D}$ with mean curvature $H$. For any $\tau > 0$, let $u_\tau: \Sigma \to \R$ be such that the normal graph $\{x + u_\tau(x) \nu(x) | x \in \Sigma\}$ is the unique entire spacelike hypersurface of constant mean curvature $\tau$ whose domain of dependence is $\mathcal{D}$. For example, $u_H = 0$. Also note that for any $\tau$, $u_\tau$ is bounded above by Lemma \ref{lem: C0 bound}. We need to show that for $\tau$ close to $H$, the function $u_\tau(x)$ is jointly analytic in $\tau$ and $x$ and $du_\tau/d\tau$ is nonvanishing, for then $(x, \tau) \mapsto x + u_\tau(x) \nu(x)$ analytically trivializes the foliation in a neighborhood of $\Sigma$.

By the main theorem of \cite{Siciak:1969aa}, to show that a function of several real variables is jointly analytic in a region, it is sufficient to show that in that region it is separately real analytic in each variable \emph{with a uniform lower bound on the radius of convergence}. To apply this directly, we should view $u_\tau(x)$ as a function on $\R^{n} \times \R$. This is easily accomplished locally in $x$ by pulling $u$ back via the exponential map of $\Sigma$. Then the elliptic regularity theory of \cite{Hopf:1932aa} implies that the functions $u_\tau(x)$ are real analytic in $x$, and also gives a lower bound -- uniform for small $\tau$ -- for the radius of convergence in $x$. We now show that for each fixed $x$, the function $u_\tau(x)$ is analytic in $\tau$ with a lower bound -- uniform in $x$ -- on its radius of convergence.

In fact, we will show that $\tau \mapsto u_\tau$ defines an analytic path in a certain H\"older space of functions on $\Sigma$. Let $r$ be the radius of convergence of this path at $\tau = H$. Since evaluation at $x$ is a bounded linear function on the H\"older space, this implies that for each $x$, $u_\tau(x)$ is an analytic function of $\tau$ with radius of convergence $r$ at $\tau = H$. We now define the H\"older space.

For any function $u \in C^{\infty}(\Sigma)$, and any nonnegative integer $k$ and $\alpha \in (0,1)$, define the global $(k,\alpha)$-H\"older norm of $u$ by
\[
|u|_{k,\alpha;\Sigma} = \max_{j \leq k} \left(\sup_{x \in \Sigma} 
|\nabla^j u(x)|\right) + \sup_{\substack{x, y \in \Sigma}} \left(
\frac{|\nabla^k u(x) - P_{y,x}\nabla^k u(y)|}{\mathrm{dist}(x,y)^\alpha}\right)
\] 
where $\nabla^k u$ is the $k$-tensor $\nabla_{i_1}\cdots\nabla_{i_k} u$, $P_{y,x}$ is the parallel transport along the unique geodesic from $y$ to $x$, and $\mathrm{dist}(x,y)$ is the intrinsic distance on $\Sigma$. The completion of the subspace of $C^\infty(\Sigma)$ on which this norm is finite is the Banach space $C^{k,\alpha}(\Sigma)$. We have used the uniqueness of the geodesic between any two points of $\Sigma$ only for convenience; if we define the supremum only over points $x$ and $y$ that are, say, within a distance one of each other, we get an equivalent norm since if $\mathrm{dist}(x,y) > 1$, then
\begin{equation}
\begin{split} \label{eqn: local Holder}
\frac{|\nabla^k u(x) - P_{y,x}\nabla^k u(y)|}{\mathrm{dist}(x,y)^\alpha} \leq 2 |u|_{k,0;\Sigma}.
\end{split}
\end{equation}

Fix $\alpha \in (0,1)$. Let $\Omega \subset C^{2,\alpha}(\Sigma)$ be an open neighborhood of 0 consisting of functions $u$ whose normal graph is still a spacelike surface. Such a neighborhood exists because the second fundamental form of $\Sigma$ is globally bounded. Define the mean curvature operator $\mathcal{H}: \Omega \to C^{0,\alpha}(\Sigma)$ which sends a function $u$ to the function which gives the mean curvature of the normal graph of $u$ at a point $x$ in $\Sigma$. So for example $\mathcal{H}(u_\tau) = \tau$ and $\mathcal{H}(0) = H$. Since $\mathcal{H}$ is an algebraic function of $u$, its derivatives, and the (analytic) second fundamental form of $\Sigma$, it is an analytic map between Banach manifolds. By Lemma \ref{lem: L inverse} below, it has an analytic inverse in an open neighborhood of the constant function $H$ in $C^{0,\alpha}(\Sigma)$. Hence, for $\tau$ close to $H$, the functions $u_\tau$ depend analytically on $\tau$. Combined with the preceding arguments, this implies that $u_\tau(x)$ is jointly analytic in $x$ and $\tau$ for $\tau$ close to $H$. Finally, by Lemma \ref{lem: dudtau} below, $du_\tau/d\tau$ is nonvanishing, which completes the proof of analyticity of the foliation.

\begin{lemma} \label{lem: L inverse}
The operator $\mathcal{H}$ has an analytic inverse in a neighborhood of the constant function $H$.
\end{lemma}

\begin{proof}

This is a consequence of the analytic inverse function theorem once we know that the linearization of $\mathcal{H}$ at 0 is invertible.
Let $L: C^{2,\alpha}(\Sigma) \to C^{0,\alpha}(\Sigma)$ be the linearization of $\mathcal{H}$ at $u=0$. It is a standard calculation in differential geometry, equivalent to the second variation formula for area, that $L = \Delta - |\II|^2$. 

The construction of an inverse for $L$ is analogous to our proof of existence for the full nonlinear problem. For any bounded domain $K \subset \Sigma$, and any $f \in C^{0,\alpha}(\overline{K})$, we can use Schauder theory to find a solution $u^K \in C^{2,\alpha}(\overline{K})$ to the Dirichlet problem $\{L u = f, u |_{\partial K} = 0\}$. Now fix $f \in C^{0,\alpha}(\Sigma)$. Taking a relatively compact exhaustion $K_i$ of $\Sigma$, we need to show that the solutions $u^{K_i}$ to the Dirichlet problems $\{L u = f_{K_i}, u |_{\partial K_i} = 0\}$ converge along with their second derivatives to a function $u \in C^{2,\alpha}(\Sigma)$ and that furthermore 
\begin{equation}
\begin{split} \label{eqn: L inverse bounded}
|u|_{2,\alpha;\Sigma} \leq C |f|_{0,\alpha;\Sigma}
\end{split}
\end{equation}
for a constant $C$ independent of $f$. Then $Lu = f$ and the bound (\ref{eqn: L inverse bounded}) implies that $L$ is injective with bounded inverse.

From Treibergs' bounds on $|\II|$ and $|\nabla \II|$, we conclude as in Section \ref{ss:ctargument} that the metric and Christoffel symbols are uniformly bounded in normal coordinates on every intrinsic ball of radius two in $\Sigma$. This gives uniform bounds on the ellipticity constants of $L$ in normal coordinates on such balls, and also uniform bounds on the difference between covariant derivatives and partial derivatives in normal coordinates. By \cite[Theorem 6.2]{GT}, for every ball $B_2$ of radius 2 and every $u \in C^{2, \alpha}(B_2)$, we have the a priori estimate on the corresponding ball of radius one
\begin{equation}
\begin{split} \label{eqn: GT}
|u|_{2,\alpha;B_1} \leq C(|u|_{0;B_2} + |Lu|_{0,\alpha;B_2})
\end{split}
\end{equation}
for some constant $C$.

To deal with the first term on the right hand side of (\ref{eqn: GT}), we use the maximum principle. By Cauchy-Schwarz we have a positive \emph{lower} bound $n H^2 \leq |\II|^2$, so the constant functions $\pm \frac{|f|_0}{nH^2}$ are super- and sub- solutions to the problem $L u = f$. These barriers bound $|u^{K_i}|$ independently of $i$ by a multiple of the supremum of $f$. Hence for the functions $u^{K_i}$, the bound (\ref{eqn: GT}) reduces to
\begin{equation}
\begin{split} \label{eqn: uKi}
|u^{K_i}|_{2,\alpha;B_1} \leq C |f|_{0,\alpha;B_2}  \leq C |f|_{0, \alpha; \Sigma}
\end{split}
\end{equation}
for each ball $B_2 \subset K_i$.

 As a consequence, the functions $u^{K_i}$ are equicontinuous along with their first two derivatives, so we may extract a subsequence converging in $C^k_\mathrm{loc}$ to a solution $u$ on $\Sigma$ of $L u = f$. Furthermore, we retain the bound (\ref{eqn: uKi}) for each ball $B_1 \subset \Sigma$.
By Equation (\ref{eqn: local Holder}) these local bounds are enough to bound $u_\Sigma$ also in the global H\"older space $C^{2,\alpha}(\Sigma)$, also by a constant times $|f|_{0, \alpha; \Sigma}$. Hence, we have constructed a bounded inverse to $L$.
\end{proof}

\begin{lemma} \label{lem: dudtau} The derivative $du_\tau(x)/d\tau$ is strictly negative.
\end{lemma}

\begin{proof}
Differentiate the equation $\mathcal{H}(u_\tau) = \tau$ to find that $L(du_\tau(x)/d\tau) = 1$. The function 0 is a supersolution to this equation, so the solution $u$ produced in the proof of the previous lemma for $f =1$ is nonpositive, and by the strong maximum principle it is strictly negative. Note that confusingly, the CMC time $\tau$ is increasing to the past of $\Sigma$.
\end{proof}

Finally, we remark that if we think of $\tau$ as a function on $\mathcal{D}$, then it is the same as the CMC time defined in \cite{MR2904912}. Our proof clearly implies that the CMC time is analytic.

\section{Application in dimension 2 + 1} \label{sec: application}

We will now focus hereafter on the case $n=2$, that is, on surfaces in Minkowski 3-space. In this section we will prove Theorem \ref{cor graph metric} and obtain Corollaries \ref{thm:inverse min lag} and \ref{cor first fundamental form} as a consequence.

\subsection{Hyperbolic surfaces and correspondence by the normal flow}
We now recall the fundamental correspondence between hyperbolic surfaces and CMC surfaces in $\R^{2,1}$, given by the normal evolution.

Let $\sigma:S\to\R^{2,1}$ be a spacelike immersion and $G_\sigma:S\to\Hyp^2\subset\R^{2,1}$ its Gauss map. For $t\in\R$, us denote by 
$\sigma_t:S\to\R^{2,1}$
 the {normal flow} of $\sigma$ at time $t$, namely
$\sigma_t=\sigma+tG_\sigma$.

\begin{lemma} \label{lemma equidistant embedding data}
Given a spacelike immersion $\sigma:S\to\R^{2,1}$, let $\I$ be its the first fundamental form  and $B$ its shape operator. Then:
\begin{itemize}
\item The first fundamental form of $\sigma_t$ is 
\begin{equation} \label{eq fff}
\I_t=\I((\mathbbm 1+tB)\cdot,(\mathbbm 1+tB)\cdot)~.
\end{equation}
\item The shape operator of $\sigma_t$ is 
\begin{equation} \label{eq shape}
B_t=(\mathbbm 1+tB)^{-1}B~.
\end{equation}
\end{itemize}
\end{lemma}
\begin{proof}
Equation \eqref{eq fff} follows easily from the definition of $\sigma_t$, from which one obtains
\begin{equation} \label{eq: differential normal evolution}
d \sigma_t= d\sigma+t dG_\sigma=d\sigma\circ(\mathbbm 1+tB)~,
\end{equation}
and therefore the desired formula. Equation \eqref{eq shape} then follows 
by differentiating Equation \eqref{eq fff} and the fact that the second fundamental form of $\sigma_t$ can be computed by the formula
$\II_t=\frac{1}{2}\left.\frac{d}{ds}\right|_{s=t}\I_s$.
\end{proof}

\begin{remark}\label{rmk:gauss map equidistant}
As an easy consequence of the proof, we see that the Gauss map of the immersion $\sigma_t$ coincides with $G_\sigma$. In fact, for every point $p$ of $S$, $G_\sigma(p)$ is orthogonal to  $d\sigma_t(v)$ for all $v\in T_p S$ by Equation \eqref{eq: differential normal evolution}.
\end{remark}

We say a spacelike surface in $\R^{2,1}$ is \emph{hyperbolic} if its first fundamental form is a hyperbolic metric. Equivalently (by Gauss' equation), if its shape operator $B$ identically satisfies $\det B=1$.

\begin{prop} \label{prop normal flow}
Given any entire hyperbolic surface $\Sigma$ in $\R^{2,1}$, the surface $\Sigma_+=\sigma_1(\Sigma)$
is an entire convex CMC-\nicefrac{1}{2} surface with $\mathcal D(\Sigma_+)=\mathcal D(\Sigma)$.
\end{prop}
\begin{proof}
By convexity of $\Sigma$, $\Sigma_+$ coincides with the level set $T^{-1}(1)$ where $T$ is the function $T(p)=d(\Sigma,p)$. By Proposition \ref{prop:T level sets} applied to $\mathcal H=\Sigma$, it follows that $\Sigma_+$ is entire if $\Sigma$ is entire, and that the two surfaces have the same domain of dependence. We only need to show that 
$\Sigma_+$ has constant mean curvature $1/2$ if $\Sigma$ is a hyperbolic surface. By Equation \eqref{eq shape} the mean curvature of $\Sigma_{+}$ is:
$$\tr B_1=\tr((\mathbbm 1+B)^{-1}B)=\frac{\tr(B^{-1}(\mathbbm 1+B))}{\det(B^{-1}(\mathbbm 1+B))}=\frac{2+\tr B^{-1}}{1+\tr B^{-1}+(\det B^{-1})}=1$$
if $\Sigma$ is hyperbolic, for $\det B=1$ by Gauss' equation.
\end{proof}

\begin{remark} \label{remark trough}

Given any spacelike  CMC-\nicefrac{1}{2} immersion $\sigma$, one can see that    
$\sigma_{-1}$
 is the immersion of a hyperbolic surface under the additional hypothesis that $\sigma$ has positive definite second fundamental form. This hypothesis is  necessary to show that  $\sigma_{-1}$ is an immersion. For example, the past normal flow of the trough
$\Sigma=\{x_1^2-x_3^2=-1\}$ at time -1 collapses each hyperbola $\Sigma\cap \{x_1=t_0\}$ to the point $(t_0,0,0)$, and thus the image of  $\sigma_{-1}$ is not an immersed surface (actually, it is the line $\{x_2=x_3=0\}$).
\end{remark}

The trough is indeed the unique example of entire CMC in $\R^{2,1}$ whose second fundamental form is not positive definite, by Theorem \ref{thm:splitting theorem}. Nevertheless, the construction of Proposition \ref{prop normal flow} \emph{a priori} cannot be reversed straightforwardly since it is not immediately clear that the surface obtained by normal flow in the past is still entire. But this fact is true (once the exception of the trough is ruled out), as a consequence of Theorem \ref{thm: uniqueness} and the main theorem of \cite{bon_smillie_seppi}:

\begin{cor} \label{cor past flow uniqueness}
Given any entire CMC-\nicefrac{1}{2} surface $\Sigma$, either $\Sigma$ is a trough or it is the surface obtained by the time 1 normal flow of an entire hyperbolic surface.
\end{cor}
\begin{proof}
Suppose $\Sigma$ is not a trough, and let $\mathcal D=\mathcal D(\Sigma)$ be the domain of dependence of $\Sigma$. Since $\Sigma$ is not a trough, its domain of dependence is not a wedge. (See Figure \ref{fig:trough}.) Hence by \cite[Theorem A]{bon_smillie_seppi}, there exists an entire hyperbolic surface $\Sigma_-$ whose domain of dependence is $\mathcal D$. Then by Proposition \ref{prop normal flow}, the surface obtained by the time 1 normal flow from $\Sigma_-$ is an entire CMC-\nicefrac{1}{2} surface with domain of dependence $\mathcal D$, which must coincide with $\Sigma$ by the uniqueness of Theorem \ref{thm: uniqueness}.
\end{proof}

\subsection{Minimal Lagrangian maps}

The results of this section concern the realization of maps between surfaces as Gauss maps of spacelike surfaces in $\R^{2,1}$. Let us introduce the following definition.

\begin{defi}
Given two simply connected Riemannian surfaces $(S,g)$ and $(S',h')$, for $h'$ a hyperbolic metric, a smooth map $f:(S,g)\to (S',h')$ is \emph{realizable} in $\R^{2,1}$ if there exists an isometric immersion $\sigma:(S,g)\to\R^{2,1}$ and a local isometry $d:(S',h')\to\Hyp^2$ such that $d\circ f= G_\sigma$ where $G_\sigma:S\to\Hyp^2$ is the Gauss map of $\sigma$. Moreover, we say that $f$ is \emph{properly realizable} if $\sigma$ is proper.
\end{defi}

 If $f$ is realizable by an immersion $\sigma$ and local isometry $d$, then for any $\eta \in SO(2,1)$, so too is it realizable by $\eta \circ \sigma$ and $\eta \circ d$. Since any two local isometries differ by post-composition with some $\eta \in SO(2,1)$ the definition is impervious to the choice of  local isometry. 
By Proposition \ref{prop:prop embedded}, the map $f$ is properly realizable if and only if it is realizable by an embedding $\sigma$ whose image is an entire spacelike surface.

When $g$ is a hyperbolic metric, realizable maps are easily characterized in terms of the minimal Lagrangian condition.

\begin{defi}
Given two hyperbolic surfaces $(S,h)$ and $(S',h')$, a local diffeomorphism $f:(S,h)\to (S',h')$ is minimal Lagrangian if its graph is both a Lagrangian and a minimal surface in $(S\times S',h\oplus h')$.
\end{defi}

\begin{remark}Here the Lagrangian condition is meant with respect to the symplectic form $\pi^*\mathrm{dA}_h-(\pi')^*\mathrm{dA}_{h'}$, where $\mathrm{dA}_h$ is the area form of $h$ and $\pi,\pi'$ denote the projections to the first and second factor. Hence this condition is equivalent to $f$ being area-preserving.
\end{remark}

We shall apply the following characterization of minimal Lagrangian local diffeomorphisms:

\begin{lemma}[{\cite{labourieCP},\cite[Proposition 1.2.3]{jeremythesis}}] \label{lemma min lag codazzi}
Let $f:(S,h)\to (S',h)$ be a local diffeomorphism. Then $f$ is minimal Lagrangian if and only if the unique positive definite $h$-self-adjoint endomorphism $B\in\Gamma(\mathrm{End}(TS))$ such that $f^*h'=h(B\cdot,B\cdot)$ satisfies $\det B=1$ and the Codazzi condition $d^{\nabla_h}\!B=0$, where $\nabla_h$ is the Levi-Civita connection of $h$.
\end{lemma}

The following characterization of realizable maps between hyperbolic surfaces is well-known. We provide a short proof for convenience of the reader.

\begin{prop} \label{prop realizable iff min lag}
A smooth map $f:(S,h)\to (S',h')$  between simply connected hyperbolic surface is realizable in $\R^{2,1}$ if and only if it is a minimal Lagrangian local diffeomorphism.
\end{prop}
\begin{proof}
If $f$ is realized by an immersion $\sigma$ so that $G_\sigma=d\circ f$, then
it is a local diffeomorphism since the differential of the Gauss map, which coincides with $B$ under the correct identifications, satisfies $\det B=1$ and is thus non-singular.
Moreover $f^*h'=G_\sigma^*g_{\Hyp^2}=h(B\cdot,B\cdot)$ where $h$ coincides with the first fundamental form of $\sigma$ and $B$ is its shape operator. By the Gauss-Codazzi equations, $\det B=1$ and  $d^{\nabla_h}\!B=0$, hence by Lemma \ref{lemma min lag codazzi} $f$ is minimal Lagrangian.

Conversely, suppose $f$ is a minimal Lagrangian local diffeomorphism and let $B$ as in Lemma \ref{lemma min lag codazzi}. Then the pair $(h,B)$ satisfies the Gauss-Codazzi equations. By the fundamental theorem of immersed surfaces in $\R^{2,1}$, there exists an isometric immersion $\sigma:(S,h)\to\R^{2,1}$ with shape operator $B$. Now if $d:(S',h')\to\Hyp^2$ is any local isometry, then $(d\circ f)^*g_{\Hyp^2}=f^*h'=h(B\cdot,B\cdot)=G_\sigma^*g_{\Hyp^2}$. Since $S$ is connected, $d\circ f$ and $G_\sigma$ differ by post-composition with an isometry $\eta$ of $\Hyp^2$. Replacing $d$ with $\eta\circ d$  concludes the proof.
\end{proof}

\subsection{Proofs of the results}

The main theorem of this section characterizes \emph{properly} realizable minimal Lagrangian maps. Before that, a little remark to clarify the statement.

\begin{remark} \label{remark d diffeo} We observe that if a local diffeomorphism $f$ is properly realized in $\R^{2,1}$, then it is a diffeomorphism onto its image. In fact the Gauss map $G_\sigma$, which by definition equals $d\circ f$, is a diffeomorphism onto its image in $\Hyp^2$, hence $f$ is a diffeomorphism onto its image and $d$ is injective on the image of $f$. By re-defining $S'$ is thus harmless to assume that $f$ is a  diffeomorphism, which we will always do in what follows.
\end{remark}

\begin{reptheorem}{cor graph metric}
Let $f:(S,h)\to (S',h')$ be a diffeomorphism between simply connected hyperbolic surfaces. Then $f$ is properly realizable in $\R^{2,1}$ if and only if the graph of $f$ is a complete minimal Lagrangian surface in $(S\times S', h\oplus h')$. In this case, both $(S,h)$ and $(S',h')$ are isometric to straight convex domains in $\Hyp^2$.
 \end{reptheorem}
 
Recall that a \emph{straight convex domain} in $\Hyp^2$ is the interior of the convex hull of a subset of $\partial_\infty\Hyp^2$ consisting of at least 3 points.

 \begin{proof}[Proof of Theorem \ref{cor graph metric}]
{We showed in Proposition \ref{prop realizable iff min lag} that $f$ is realizable if and only if it is minimal Lagrangian. In light of Proposition \ref{prop normal flow} and Corollary \ref{cor past flow uniqueness}, $f$ is properly realizable by an immersion $\sigma$ if and only if the equidistant immersion $\sigma_1$ is entire, which by Proposition \ref{lemma complete implies entire}  and the Cheng-Yau theorem is equivalent to the completeness of the CMC first fundamental form.
Observe that when applying Corollary \ref{cor past flow uniqueness}, we used that the equidistant CMC surface is not a trough, for otherwise it could not be obtained (even locally) as the equidistant surface from a hyperbolic surface (see Remark \ref{remark trough}). }

{Now, by \eqref{eq fff} the CMC metric is identified to $\I+2\II+\III$, where $\I,\II$ and $\III$ are the first, second and third fundamental form of $\sigma$. We claim that this metric is bi-Lipschitz to the induced metric on the graph of $f$ in $(S\times S', h\oplus h')$. In fact, the latter equals $h+f^*h'$, which is to say $\I+\III$, when pulled-back by the obvious embedding $(\mathrm{id},f)$ of $S$ in the product. Young's inequality implies that $2\II\leq \I+\III$, hence $\I+2\II+\III\leq 2(\I+\III)$, whereas $\I+\III\leq \I+2\II+\III$ since by convexity $\II$ is positive definite. This shows that the CMC induced metric is complete if and only if $h+f^*h'$ is complete, which is the first part of the theorem.}

The fact that $(S',h')$ is isometric to a straight convex domain follows from \cite[Theorem E]{bon_smillie_seppi}, which shows that the image of the Gauss map of an entire hyperbolic surface (or more generally, of an entire spacelike surface of   curvature bounded above and below by negative constants) is isometric to a straight convex domain. Since $f$ is a diffeomorphism by hypothesis and $d$ is a diffeomorphism onto its image (Remark \ref{remark d diffeo}), we conclude that $d$ provides an isometry between $(S',h')$ and a straight convex domain. Replacing $f$ with $f^{-1}$, we then obtain that $(S,h)$ is isometric to a straight convex domain, since $h'+(f^{-1})^*h=(f^{-1})^*(h+f^*h')$ is complete if and only if $h+f^*h'$ is complete.
 \end{proof}
 
 \begin{remark}
Alternatively, the second statement of Theorem \ref{cor graph metric} can be proved by applying \cite[Theorem 4.8]{choitreibergs} instead of \cite[Theorem E]{bon_smillie_seppi}. Namely the statement that the image of the Gauss map of any entire CMC hypersurface $\Sigma$ in $\R^{n,1}$ is either contained in a totally geodesic subspace of $\Hyp^n$, or the (non-empty) interior of convex hull of a subset of the boundary of $\Hyp^n$. When $n=2$, this means that either $\Sigma$ is a trough, which is not possible in our setting, or the image of the Gauss map of $\Sigma$ is a straight convex domain. Since the Gauss maps of $\sigma$ and $\sigma_1$ coincide (Remark \ref{rmk:gauss map equidistant}), one infers again that the image of $G_\sigma$ is a straight convex domain and then concludes the proof analogously.
\end{remark}

\begin{remark} The CMC-\nicefrac{1}{2} metric $\I + 2\II + \III$ is conformal to the induced metric $\I + \III$ on the graph of $f$. Since the graph of $f$ is minimal, the two projections to $S$ and $S'$ are each harmonic, with opposite Hopf differential. Since the property of being harmonic is conformally invariant in the domain, we see that the Gauss map of the CMC surface and the projection to the corresponding hyperbolic surface are harmonic maps from the CMC surface, with opposite Hopf differential. Thus, the CMC surface realizes the well-known decomposition of minimal Lagrangian maps in terms of a pair of harmonic maps.
\end{remark}

The last argument of the proof of Theorem \ref{cor graph metric} shows:

\begin{repcor}{thm:inverse min lag}
Let $f:(S,h)\to (S',h')$ be a minimal Lagrangian diffeomorphism between simply connected hyperbolic surfaces. Then $f$ is properly realizable in $\R^{2,1}$ if and only if $f^{-1}$ is properly realizable in $\R^{2,1}$.
\end{repcor}

We conclude with the following corollary which characterizes the first fundamental forms of entire hyperbolic surfaces.

\begin{repcor}{cor first fundamental form}
Any entire hyperbolic surface in $\R^{2,1}$ is intrinsically isometric to a straight convex domain. Conversely any straight convex domain in $\Hyp^2$ can be isometrically embedded  in $\R^{2,1}$ with image an entire surface.
\end{repcor}
\begin{proof}
  The first statement is contained in Theorem \ref{cor graph metric}.
To prove the second statement, by \cite[Theorem A, Theorem E]{bon_smillie_seppi}, given any straight convex domain $\Omega\subseteq\Hyp^2$, there exists an entire hyperbolic surface $\Sigma$ in $\R^{2,1}$ having $\Omega$ as image of the Gauss map. (In fact, there is one such surface for every choice of a lower-semicontinuous function on $\partial_\infty\Hyp^2$ which is finite on $\overline{\Omega}\cap\partial_\infty\Hyp^2$.) Let $f:\Sigma\to\Omega$ be the Gauss map of such surface. By Theorem \ref{cor graph metric}, $f^{-1}$ is realized in $\R^{2,1}$, which implies that there exists an isometric embedding of $\Omega$ onto an entire hyperbolic surface.
\end{proof}

\bibliographystyle{alpha}
\bibliographystyle{ieeetr}
\bibliography{../biblio_new.bib}

\end{document}